\documentclass[USenglish]{article}

\usepackage{fullpage}

\usepackage{amsmath,amsfonts,amsthm,amssymb}

\newtheorem{theorem}{Theorem}[section]

\newtheorem{proposition}[theorem]{Proposition}
\newtheorem{remark}[theorem]{Remark}

\newtheorem{assumption}[theorem]{Assumption}

\usepackage{color,xcolor}
\usepackage{ifpdf}

\usepackage{psfrag}
\usepackage{graphicx,graphics,subfigure}
\usepackage{url}

 \newcommand{\R}{\mathbb R}

\newcommand{\PP}{\mathbb P}

\newcommand{\normal}{\mathbf{N}}

 \newcommand{\KK} {\mathcal{K}}

 \newcommand{\diag}{\mathop{\mathrm{diag}}}

\newcommand{\bbf}{{\mathbf {f}}}

\newcommand{\bn}{{\mathbf {n}}}

\newcommand{\Bf}{\mathbf{f}}
\newcommand{\dpar}[2]{\dfrac{\partial #1}{\partial #2}}

\newcommand{\bu}{\mathbf{u}}
\newcommand{\bv}{\mathbf{v}}

\newcommand{\BV}{\underline{\mathbf{V}}}

\usepackage{mathrsfs}
\DeclareMathAlphabet{\mathpzc}{OT1}{pzc}{m}{it}
\newcommand{\bV}{\mathbf{\mathpzc{V}}}

\newcommand{\bF}{\mathbf{f}}
\newcommand{\bbF}{\mathbf{\mathcal{F}}}

\newcommand{\bbu}{\mathbf{u}}
\newcommand{\bbg}{\mathbf{g}}
\newcommand{\hbbg}{\hat{\mathbf{g}}}
\newcommand{\hbbf}{\hat{\mathbf{f}}}

\newcommand{\bx}{\mathbf{x}}
\newcommand{\by}{\mathbf{y}}

\newcommand{\red}[1]{{#1}}

\newcommand{\Kp}{{|K}}
\newcommand{\Km}{{|K^-}}

\usepackage{hyperref}
\hypersetup{
pdftitle={A general framework to construct schemes that satisfy additional conservation relations.\\
Application to entropy conservative and entropy dissipative schemes}
pdfauthor={R. Abgrall},
pdfpagemode=UseOutlines,
linkbordercolor=0 0 0,
linkcolor=red,
citecolor=blue,
colorlinks=true,
bookmarks = true
}
\begin{document}
\title{A general framework to construct schemes satisfying additional conservation relations.\\
Application to entropy conservative and entropy dissipative schemes}
\author{R. Abgrall\\
Institute of Mathematics,
University of Zurich, Switzerland}
\date{\today}
\maketitle

\begin{abstract}
We are interested in the approximation of a steady hyperbolic problem. In some cases, the solution can satisfy an additional conservation relation, at least when it is smooth. This is the case of an entropy.  In this paper, we show, starting from the discretisation of the original PDE, how to construct a scheme that is consistent with the original PDE and the additional conservation relation. Since one interesting example is given by the systems endowed by an entropy, we provide one explicit solution, and show that the accuracy of the new scheme is at most degraded by one order. In the case of a discontinuous Galerkin scheme and a Residual distribution scheme, we show how not to degrade the accuracy. This improves the recent results obtained  in \cite{ShuEntropy,Mishra,Gassner,Zingg} in the sense that no  particular constraints are set on quadrature formula and that a priori maximum accuracy can still be achieved. We study the behaviour of the method on a non linear scalar problem. However, the method is not restricted to scalar problems.
\end{abstract}
\section{Introduction}
In this paper, we are interested in the approximation of non-linear hyperbolic problems.
 To make things more precise, our target are the  Euler equations in the compressible regime, other examples are the MHD equations. The case of parabolic problems in which the elliptic terms play an important role only in some area of the computational domain, such as the Navier-Stokes equations in the compressible regime, or the resistive MHD equations, can be dealt with in a similar way. 
 
Let $\mathcal{D}$ an open subset of $\R^p$ and $d$\; $C^1$ functions $\bbf_j, j=1,\ldots, d$ defined on $\mathcal{D}$. We consider the hyperbolic problem, with $\bu:\Omega\subset\R^d\rightarrow \mathcal{D}\subset \R^p$, defined by 
\begin{subequations}\label{eq1}
\begin{equation}
\label{eq1:1}
\text{ div }\bbf(\bu)=\sum\limits_{j=1}^d\dpar{\bbf_j}{x_j}(\bbu)=0
\end{equation}
subjected to
\begin{equation}\label{eq1:2}
\red{\bigg( \nabla_\bu \bbf(\bu)\cdot \bn(\bx),0\bigg)^-}(\bu-\bu_b)=0 \text{ on }\partial\Omega.
\end{equation}
In  \eqref{eq1:2}, $\bn(\bx)$ is the outward unit vector at $\bx\in \partial \Omega$ and $\bu_b$ is a  regular enough function.
{We  have also used standard notations: if $A$ is a $p\times p$ diagonalisable matrix in $\R$, $A=L\Lambda R$ where $L$ and $R$ are two $p\times p$ matrices with $LR=\text{Id}_{p\times p}$ and $\Lambda=\diag (\lambda_i)_{i=1, \ldots, p}$ diagonal, then
$$A^\pm=L\Lambda^\pm R \text{ where } \Lambda^\pm=\diag (\lambda_i^\pm)_{i=1, \ldots, p}$$
and, for $x\in \R$, $x^+=\max(x, 0)$ and $x^-=\min(x,0)$.}
\end{subequations}
  The weak formulation of \eqref{eq1} is: $\bu\in L^2(\Omega)^d\cap L^\infty(\Omega)^d$ is a weak solution of \eqref{eq1} if $\bbu\in \mathcal{D}$ and for any $\varphi\in C^1_0(\Omega)^d$, 
\begin{equation}
\label{weak:eq1}
-\int_\Omega \nabla \varphi \cdot \bbf(\bu)\; d\bx+\int_{\partial \Omega} \varphi \big ( \mathbf{\mathcal{F}}_\bn(\bu,\bu_b)-\bbf(\bu)\cdot\bn\big ) \; d\gamma=0
\end{equation}
where $\mathbf{\mathcal{F}}_\bn$ is a flux that is almost everywhere the upwind flux:
$$\mathbf{\mathcal{F}}_\bn(\bu,\bu_g)=\big( \nabla_\bu \bbf(\bu)\cdot \bn(\bx)\big)^+ \bbu+\big( \nabla_\bu \bbf(\bu)\cdot \bn(\bx)\big)^-\bbu_g$$

 In general, schemes are constructed starting from  the conservation relation. In some cases, such as when the solution is smooth enough, it appears that the \red{ solution of \eqref{eq1} also satisfies}  an additional conservation relation. There is no reason that the initial discretisation will also be consistent with the new conservation relation.  A typical example is the entropy.
If $U$, defined on $\mathcal{D}$, is a convex and $C^1$ function and if there exists $g$, a $C^1$ function, $g=(g_1, \ldots, g_d):\mathcal{D}\rightarrow \R^d$,  such that: for all $j=1, \ldots, d$,
\begin{subequations}
\begin{equation}\label{edp:entropy:1}\nabla_\bu U \cdot \nabla_\bu \bbf_j=\nabla_\bu g_j,\end{equation}
then when $\bu$ is smooth we get the additional conservation relation
\begin{equation}\label{edp:entropy:2}\text{ div }\bbg(\bu)=\sum\limits_{j=1}^d\dpar{g_j}{x_j}(\bu)=0.
\end{equation}
\end{subequations}
An entropy solution of \eqref{eq1} will satisfy $\text{ div }\bbg(\bu)\leq 0$
in the sense of distributions. 

Hence a rather natural question is: given a scheme for \eqref{eq1}, how can we modify it so that the modified scheme is consistent with an additional conservation relation, i.e. in the case of an entropy, how to modify it so that it is also compatible with an entropy inequality? This is not a new question: there has been lots of work on that particular question, see \cite{TadmorB,TadmorVieux,TadmorActa,JiangShu,bouchut,LeFlochMercier}, and more recently see \cite{HouLiu,Ismail,Zingg,Gassner,Mishra,Guermond,ShuEntropy,Oefner,Carpenter1,Nordstrom1,hiltebrand14:_entrop_galer,svard,zbMATH06599705,zbMATH06799652,GASSNER2016291,chandrasekar} and references therein for an un-complete list of contributions. One of the salient aspects \red{of research } on this problem is a focus on the discrete version of the schemes: \red{for example, when  some integral appear in the formulation, as in finite element method, quadrature formula are needed, and the integrals are never computed exactly}. \red{The central question is then how to translate, from a discrete point of view, the 'continuous' differential relations \eqref{edp:entropy:1} and the algebra associated to them at the discrete level.} Up to our knowledge, entropy stable schemes can be rigorously designed only for special quadrature formula. The objective of this paper is to provide a different point of view within the Residual Distribution framework. \red{In particular, we show  that we can consider general quadrature formula.}

\red{The outline of the paper is as follows.} In a first part, we recall the class of schemes (nicknamed as Residual distribution schemes or RD or RDS for short)  we are interested in, and show their link with more classical methods such as finite element ones. Then we recall  a condition that guaranties that the numerical solution will converge to a weak solution of the problem, and a second one about entropy condition, and recall a systematic way to check the formal accuracy of the scheme. In the third part, we give  a general recipe that show how to modify a scheme in order to fullfil an additional conservation law. The fourth part is devoted 
 to a discussion on the entropy conservation: we show explicitly how to modify a given scheme. The next section is devoted to the construction of entropy dissipative schemes, and we provide numerical examples. A conclusion follows.
\section{Notations}
 From now on, we assume that $\Omega$ has a polyhedric boundary in $\R^d$, $d=1,2,3$.  This simplification is by no mean essential. We denote by $\mathcal{E}_h$ the set of internal edges/faces of \red{ a triangulation} $\mathcal{T}_h$, and by $\mathcal{F}_h$ those contained in $\partial \Omega$.  $\KK$ stands either for an element $K$ or a face/edge $e\in \mathcal{E}_h\cup \mathcal{F}_h$. The boundary faces/edges are denoted by $\Gamma$.  The mesh is assumed to be shape regular, $h_K$ represents the diameter of the element $K$. Similarly, if $e\in \mathcal{E}_h\cup \mathcal{F}_h$, $h_e$ represents its diameter. \red{Throughout the text, we will restrict ourselves to the case where the elements $K$ are simplicies, mostly for simplicity reasons.}

 Throughout this paper, we follow Ciarlet's definition \cite{ciarlet,ErnGuermond} of a finite element approximation: \red{for any element $K$,  we have a set of linear forms  $\Sigma_K$  acting on the set $\PP^k$ of polynomials of degree $k$ such that the linear mapping}
 $$q\in \PP^k\mapsto \big (\sigma_1(q), \ldots, \sigma_{|\Sigma_K|}(q)\big )$$
 is one-to-one. The space $\PP^k$ is spanned by the basis function $\{\varphi_{\sigma}\}_{\sigma\in \Sigma_K}$  defined by
 $$\forall \sigma,\, \sigma',  \sigma(\varphi_{\sigma'})=\delta_\sigma^{\sigma'}.$$
  {We have in mind either the Lagrange interpolation where the degrees of freedom are associated to \red{the Lagrange  points} in $K$, or other type of polynomials approximation such as using B\'ezier polynomials where we will also do the same geometrical identification. More specifically, if $\{\mathbf{A}_1, \ldots, \mathbf{A}_{d+1}\}$ are the vertices of $K$ (i.e $K$ is the convex hull of the vertices $\{\mathbf{A}_1, \ldots, \mathbf{A}_{d+1}\}$), the Lagrange points of degree $k$ are defined by their barycentric coordinates 
    \begin{equation}
    \label{lagrange}L_\alpha=\sum_{l=1}^{d+1} \frac{i_l}{d+1}\mathbf{A}_l, \quad \alpha=(i_1, \ldots, i_{d+1})
    \end{equation}
  with $i_l$ positive integers such that $i_1+\ldots +i_{d+1}=k$

Any point $\mathbf{M}$ of $\R^d$ can be uniquely represented by its barycentric coordinates $\{\lambda_1(\mathbf{M}), \ldots, \lambda_{d+1}(\mathbf{M})\}$ with
  $$\mathbf{M}=\sum_{l=1}^{d+1} \lambda_l(\mathbf{M}) \mathbf{A}_l, \qquad \sum_{l=1}^{d+1}\lambda_l(\mathbf{M})=1.$$
  Note that $\lambda_l(\mathbf{M})\geq 0$ if and only if $\mathbf{M}\in K$.
  The B\'ezier polynomials of degree $k$,  associated to the multi-index $\alpha=(i_1, \cdots, i_{d+1})$, $i_l$ positive integer, $i_1+\ldots +i_{d+1}=k$  as in \eqref{lagrange},
  are the polynomials 
  $$B_\alpha=\dfrac{k!}{i_1! \cdots  i_{d+1}!} \lambda_1^{i_1}\ldots \lambda_{d+1}^{i_{d+1}}.$$
  They sum up to unity, are positive on $K$, are interpolatory on the vertices only, span $\PP^k$ as well as the Lagrange polynomials of degree $k$. For the Lagrange interpolation of degree $k$, the set of degrees of freedom is the set of Lagrange points of degree $k$, i.e. the $L_\alpha$ where $\alpha$ is as in \eqref{lagrange}. For the B\'ezier approximation, we make the geometrical identification between the muti-index $\alpha$ that defines the polynomial $B_\alpha$ and the Lagrange point $L_\alpha$, though the B\'ezier polynomials are not interpolatory in general. However, the B\'ezier polynomials of degree $k$ span $\PP_k$, as well as the Lagrange polynomials of degree $k$.}
  
 Considering all the elements covering $\Omega$, the set of degrees of freedom is denoted by $\mathcal{S}$ and a generic degree of freedom  by $\sigma$. We note that for any $K$, 
 $$\forall \bx\in K, \quad \sum\limits_{\sigma\in K}\varphi_\sigma(\bx)=1.$$
 For any element $K$, $\#K$ is the number of degrees of freedom in $K$. If $\Gamma$ is a face or a boundary element, $\#\Gamma$ is also the number of degrees of freedom in $\Gamma$.

  The solution is sought for in the space  $V_h^k$  (or $V_h$ for short since the degree $k$ will be assumed to be constant) defined by, setting
$$\mathcal{V}_h=\bigoplus_K\{ \bv\in L^2(K), \bv_{|K}\in \PP^k\}.$$
We will consider two cases $V_h=\mathcal{V}_h\cap C^0(\Omega)$ in which case the triangulation needs to be conformal, and $V_h=\mathcal{V}_h$ where the continuity requirement is released: the triangulation does not need to be conformal anymore.

Throughout the text, we need to integrate functions. This is done via quadrature formula, and the symbol $\oint$ used in volume integrals
\begin{equation}\label{quad:K}
\oint_K v(\bx)\; d\bx:=|K|\sum_{q}\omega^K_q v(\bx_q^K)
\end{equation}
or boundary/face integrals
\begin{equation}\label{quad:f}
\oint_{\partial K} v(\bx)\; d\gamma:=|f|\sum_{q}\omega^f_q v(\bx_q^f)
\end{equation}
means that these integrals are done via user defined numerical quadratures.

 If $e\in \mathcal{E}_h$, represents any  \emph{internal} edge, i.e. $e\subset K\cap K^+$ for two elements $K$ and $K^+$,  we define for any function $\psi$ the jump  $[\nabla \psi ]=\nabla \psi_{|K}-\nabla \psi_{| K^+}$.
 Similarly, $\{\bv\}=\tfrac{1}{2}\big (\bv_{|K}+\bv_{|K^+}\big )$.
 
{ If $\mathbf{x}$ and $\mathbf{y}$ are two vectors of $\R^q$, for $q$ integer, $\langle \bx,\by\rangle$ is their scalar product. In some occasions, it can also be denoted as $\bx\cdot\by$ or $\bx^T\by$. Similarly, if $\bx=(x_1, \ldots, x_d)$ is a vector and $\bf=(\bf_1, \ldots, \bf_d)$ with $\bf_j\in \R^m$, then 
 we denote
 $$\bx\cdot \bbf=\bx^T\bbf=\sum_{i=1}^d x_i\bbf_i.$$
 
 Last, $v^h\in V^h$ will denote any test function, $\bu$ will be the conserved quantities defined  in \eqref{eq1},  $\bV$ will be the entropy variable and $\bV^h\in V^h$ will be its interpolant in $\PP^k$. Depending on the context, $\bu^h\in V^h$ will denote either an approximation of $\bu$ or $\bu(\bV^h)$. Similarly, $\pi_h(\bu)\in V^h$ is the interpolant of $\bu$ or $\bu(\pi_h(\bV))$, depending on the context.}

     \section{Schemes, conservation, entropy dissipation}
\subsection{Schemes}
{ In order to integrate the steady version of \eqref{eq1} on a domain $\Omega\subset \R^d$ with the boundary conditions \eqref{eq1:2}, on each element $K$ and any degree of freedom $\sigma\in \mathcal{S}$ belonging to $K$,  we define residuals $\Phi_\sigma^K(\bu^h)$. Following \cite{abgrallLarat,abgralldeSantisSISC}, they are assumed to satisfy the following conservation relations:
For any element $K$,  and any $\bu^h\in V^h$,
\begin{subequations}
\label{conservation}
\begin{equation}
\label{conservation:K}
\sum\limits_{\sigma\in K}\Phi_\sigma^K(\bu^h)=\oint_{\partial K}\hbbf_\bn(\bu^h_\Kp,\red{\bu^{h}_\Km})\; d\gamma.
\end{equation}
\red{where $\red{u_\Kp}$ is the restriction of $\bu^h$ in the element $K$\, while \red{$\bu^{h}_\Km$} is the restriction of $\bu^h$  on the other side of the local edge/face of $K$. In addition, $\hat{\bbf}_\bn$ is a consistant numerical flux, i.e. $\hbbf(\bu,\bu)=\bbf(\bu)\cdot\bn$. }
Similarly, we consider residuals on the boundary elements $\Gamma$. On any such $\Gamma$, for any degree of freedom $\sigma\in \mathcal{S}\cap \Gamma$, we consider boundary residuals $\Phi_\sigma^\Gamma(u^h)$ that will satisfy the conservation relation
\begin{equation}
\label{conservation:Gamma}
\sum\limits_{\sigma\in \Gamma}\Phi_\sigma^\Gamma(\bu_h)=\oint_{\Gamma}\big ( \hat{\bbf}_\bn(\bu^h_{|\Gamma}, \bu_b)-\bF(\bu^h)\cdot \bn\big ) \; d\gamma
\end{equation}
\end{subequations}
\red{where, for $\bx\in \Gamma$,  $\bu^h_{|\Gamma}(\bx)=\lim\limits_{\by\rightarrow \bx, \by\in \Omega}\bu^h(\by)$.}
Once this is done, the discretisation of \eqref{eq1} is achieved via: for any $\sigma\in \mathcal {S}$,
\begin{equation}
\label{RD:scheme}
\sum\limits_{K\subset\Omega, \sigma\in K}\Phi_\sigma^K(\bu^h)+\sum\limits_{\Gamma\subset\partial\Omega, \sigma\in \Gamma}\Phi_\sigma^\Gamma(\bu^h)=0.
\end{equation}
In \eqref{RD:scheme}, the first term represents the contribution of the internal elements. The second exists if $\sigma\in \Gamma$ and it represents the contribution of the boundary conditions.

In fact, the formulation \eqref{RD:scheme} is very natural. Consider a variational formulation of the steady version of \eqref{eq1}:
$$\text{ find } \bu^h\in V^h \text{ such that for any } \mathbf{v}^h\in V^h,  a(\bu^h,\bv^h)=0.$$
Let us show on  three examples that this variational formulation leads to \eqref{RD:scheme}. They are
\begin{itemize}
\item The SUPG \cite{Hughes1} variational formulation,  with $\bu^h, \bv^h\in V^h=\mathcal{V}^h\cap C^0(\Omega)$:
\begin{equation}
\label{SUPG:var}
\begin{split}
a(\bu^h,\bv^h)&:=-\oint_\Omega \nabla \bv^h\cdot \bF(\bu^h)\; d\bx+\sum\limits_{K\subset \Omega}h_K\oint_K\big [ \nabla \bF(\bu^h)\cdot\nabla \bv^h\big ] \; \tau_K\; \big [\nabla\bF(\bu^h)\cdot \nabla \bu^h\big ] d\bx\\
&\qquad +\oint_{\partial \Omega} \bv^h \big (\bbF_\bn(\bu^h,\bu_b)-\bF(\bu^h)\cdot\bn\big ) \;d\gamma .
\end{split}
\end{equation}
and $\tau_K>0$.
\item The Galerkin scheme with jump stabilization,  see \cite{burman} for details.  We have
\begin{equation}
\label{burman:var}
\begin{split}
a(\bu^h,\bv^h)&:=-\oint_\Omega \nabla \bv^h\cdot \bF(\bu^h)\; d\bx+\sum\limits_{e \subset \Omega}\theta_e h_K^2\oint_e \big [ \nabla \bv^h \big ]\cdot \big [ \nabla \bu^h\big ] \; d\gamma \\
&\qquad +\oint_{\partial \Omega} \bv^h \big (\bbF_\bn(\bu^h,\bu_b)-\bF(\bu^h)\cdot\bn\big ) \; d\gamma .
\end{split}
\end{equation}
 Here,  $\bu^h, \bv^h\in V^h=\mathcal{V}^h\cap C^0(\Omega)$, and $\theta_e$ is a positive parameter.
\item The discontinuous Galerkin formulation: we look for $\bu^h, \bv^h\in V^h=\mathcal{V}^h$ such that
\begin{equation}\label{DG:var}
a(\bu^h,\bv^h):=\sum\limits_{K\subset \Omega}\bigg ( -\oint_K\nabla\bv^h\cdot\bbf(\bu^h) d\bx+\oint_{\partial K}\bv^h\cdot \hat{\bbf}_\bn(\bu^{h}_\Kp,\red{\bu^{h}_\Km}) \;d\gamma \bigg ).
\end{equation}
In \eqref{DG:var}, the boundary integral is a sum of integrals on the faces of $K$, and here for any face of $K$
 $\red{\bu^{h}_\Km}$ represents the approximation of $\bu$ on the other side of that face in the case of internal elements, and $\bu_b$ when that face is on $\partial \Omega$.  In \eqref{DG:var}, $\hbbf$ is a consistent numerical flux. Note that to fully comply with \eqref{RD:scheme}, we should have defined for boundary faces $\red{\bu^{h}_\Km}=\bu^h_\Kp$, and then \eqref{DG:var} is rewritten as 
 \begin{equation}\label{DG:var:2}
 \begin{split}
a(\bu^h,\bv^h)&:=\sum\limits_{K\subset \Omega}\bigg ( -\oint_K\nabla\bv^h\cdot\bbf(\bu^h) d\bx+\oint_{\partial K}\bv^h \cdot\, \hat{\bbf}_\bn(\bu^{h}_\Kp,\red{\bu^{h}_\Kp}) \; d\gamma \bigg )\\
&\qquad  \qquad +\sum\limits_{\Gamma\subset\partial\Omega}\oint_{\Gamma}\bv^h\cdot\bigg ( \hat{\bbf}_\bn(\bu^h_\Gamma,\bu_b)-\bbf(\bu^h)\cdot \bn \bigg )\; d\gamma.
\end{split}
\end{equation}
with by abuse of language, $\hbbf=\bbF$ on the boundary edges.
\end{itemize}
In the SUPG, Galerkin scheme with jump stabilisation or the DG scheme, the boundary flux can be chosen different from $\bbF$. This can lead to boundary layers if these flux are not "enough" upwind, but we are not interested in these issues here.

Using the fact that the basis functions that span $V_h$ have a \emph{compact} support, then each scheme can be rewritten in the form\eqref{RD:scheme} with the following expression for the residuals:
    \begin{itemize}
    \item For the  SUPG scheme \eqref{SUPG:var}, the  residual are defined by
    \begin{equation}\label{SUPG}\Phi_\sigma^K(\bu^h)=\oint_{\partial K}\varphi_\sigma \bF(\bu^h)\cdot \bn \; d\gamma -\oint_K \nabla \varphi_\sigma\cdot \bF(\bu^h) \; d\bx+h_K
    \oint_K \bigg (\nabla_\bu\bF(\bu^h)\cdot \nabla \varphi_\sigma \bigg )\tau_K \bigg (\nabla_\bu\bF(\bu^h)\cdot \nabla \bu^h \bigg )\;d\bx\end{equation}
    with $\tau_K>0$.
    \item For the Galerkin scheme with jump stabilization \eqref{burman:var}, the residuals are defined  by:  
        \begin{equation}\label{burman}\Phi_\sigma^K(\bu^h)=\oint_{\partial K}\varphi_\sigma \bF(\bu^h)\cdot \bn\; d\gamma -\oint_K \nabla \varphi_\sigma\cdot \bF(\bu^h)\; d\bx +
    \sum\limits_{e \text{ faces of }K} \theta_e\; h_e^2 \oint_{\partial K} [\nabla \bu]\cdot [\nabla \varphi_\sigma]\; d\gamma\end{equation}
    with $\theta_e>0$.
    Here, since the mesh is conformal, any internal edge $e$ (or face in 3D) is the intersection of the element $K$ and an other element denoted by $K^+$.
    \item For the discontinuous Galerkin scheme,
    \begin{equation}\label{DG}
    \Phi_\sigma^K(\bu^h)=-\oint_K\nabla\varphi_\sigma\cdot\bbf(\bu^h) d\bx+\oint_{\partial K}\varphi_\sigma\cdot \hat{\bbf}_\bn(\bu^{h}_\Kp,\red{\bu^{h}_\Km}) \; d\gamma
    \end{equation}
    using the second definition of $\red{\bu_\Km^{h}}$.
    \item The boundary residuals are 
  \begin{equation}
  \label{boundary}
  \Phi_\sigma^\Gamma(\bu^h)=  \oint_{\Gamma}\varphi_\sigma\big ( \hat{\bbf}_\bn(\bu^h_\Gamma, \bu_b)-\bF(\bu^h)\cdot \bn\big )\; d\gamma
  \end{equation}
  \end{itemize}
All these residuals satisfy the relevant conservation relations, namely \eqref{conservation:K} or \eqref{conservation:Gamma}, depending if we are dealing with element residuals or boundary residuals.

\bigskip

For now, we are just rephrasing classical finite element schemes into a purely numerical framework. However, considering  the pure numerical point of view and forgetting the variational framework, we can go further and define schemes that have no clear variational formulation. 
These are  the limited  Residual Distribution Schemes, see \cite{abgrallLarat,abgralldeSantisSISC}, namely
    \begin{equation}
    \label{schema RDS}\Phi_\sigma^K(\bu^h)=\beta_\sigma \oint_{\partial K}\bF(\bu^h)\cdot \bn\; d\gamma
    \end{equation}
    or 
    \begin{equation}
    \label{schema RDS SUPG}\Phi_\sigma^K(\bu^h)=\beta_\sigma \oint_{\partial K}\bF(\bu^h)\cdot \bn\; d\gamma+h_K\theta_S
    \oint_K \bigg (\nabla_\bu\bF(\bu^h)\cdot \nabla \varphi_\sigma \bigg )\tau_K \bigg (\nabla_\bu\bF(\bu^h)\cdot \nabla \bu^h \bigg )\;d\bx
    \end{equation}
or
\begin{equation}
\label{schema RDS jump}\Phi_\sigma^K(\bu^h)=\beta_\sigma \oint_{\partial K}\bF(\bu^h)\cdot \bn\; d\gamma+
    \theta_J \;h_K^2 \oint_{\partial K} [\nabla \bu^h]\cdot [\nabla \varphi_\sigma]\; d\gamma.\end{equation}
    or, combining the two kind of stabilisations
    \begin{equation}
\label{schema RDS jump Stream}\Phi_\sigma^K(\bu^h)=\beta_\sigma \oint_{\partial K}\bF(\bu^h)\cdot \bn\; d\gamma+
    \theta_J \;h_K^2 \oint_{\partial K} [\nabla \bu^h]\cdot [\nabla \varphi_\sigma]\; d\gamma +h_K\theta_S
    \oint_K \bigg (\nabla_\bu\bF(\bu^h)\cdot \nabla \varphi_\sigma \bigg )\tau \bigg (\nabla_\bu\bF(\bu^h)\cdot \nabla \bu^h \bigg )\;d\bx\end{equation}
   where the parameters $\beta_\sigma$ are defined to guarantee conservation,
   $$\sum\limits_{\sigma\in K} \beta_\sigma=1$$
   and such that \eqref{schema RDS SUPG} without the streamline term and \eqref{schema RDS jump} without the jump terms satisfy a discrete maximum principle. The parameters $\tau_K$, $\theta_J$ and $\theta_S$ \red{in \eqref{schema RDS SUPG}, \eqref{schema RDS jump} and \eqref{schema RDS jump Stream}} are positive real numbers. The streamline term and jump term are introduced because one can easily see that spurious modes may exist, but their role is very different compared to \eqref{SUPG} and \eqref{burman} where they are introduced to stabilize the Galerkin scheme: if formally the maximum principle is violated, experimentally the violation is extremely small, if not existant. See \cite{energie,abgrallLarat} for more details. 
   
   A similar construction can be done starting from a discontinuous Galerkin scheme without non-linear stabilisation such as limiting, has been applied. 
   
   The non-linear stability is provided by the coefficient $\beta_\sigma$ which is a non-linear function of $\bu^h$.  Possible values of $\beta_\sigma$ are described in the appendix \ref{RDS}.
   
  }
\subsection{Conservation}\label{sec:conservation}
{
It is time to state the geometrical assumptions we make on the quadrature formula: 
\begin{assumption}[Assumption on the quadrature formula]\label{assump:quad}
In \eqref{conservation:Gamma} and \eqref{conservation:K}, numerical integrations are made on the faces of the elements of $\mathcal{T}_h$, boundary of $\Omega$ included. Hence when we write, for an element $K$
$\oint_{\partial K} \psi \cdot \bn \; d\gamma$, or for a face on $\partial \Omega$ $\oint_{\Gamma} \psi \; d\gamma$, we are adding contributions  'living' on the squeleton of $\mathcal{T}_h$. We request that the quadrature formula depend only on the squeleton elements, so that if $f=K_1\cap K_2$ then
\begin{equation}
\label{quadrature}
\oint_{f\cap \partial K_1} \psi \cdot \bn\; d\gamma+\oint_{f\cap \partial  K_2} \psi \cdot \bn\; d\gamma=0.
\end{equation}
In the previous relation, the first integral is a numerical formula on $f$ seen from $K_1$, and the second term is an integral seen from $K_2$. Since the  normal  are opposite, this means that the geometrical location of the quadrature points are the same on $f$ seen from $K_1$ and $f$ seen from $K_2$. If $f$ is on the boundary of $K$, this is also the intersection of an element $K$ and the boundary of $\Omega$, and again, the quadrature points have the same geometrical locations whatever be the way we consider $f$.
\end{assumption}
In practice, this is not a restriction at all since quadrature points are defined by their barycentric coordinates and symmetric with respect to permutations of the vertices that respect orientation.

\bigskip

From \eqref{RD:scheme}, using the conservation relations \eqref{conservation:Gamma} and \eqref{conservation:K}, we obtain for any $\bv^h\in V^h$,
$$\bv_h=\sum_{\sigma \in \mathcal{S}} \bv_\sigma \varphi_\sigma,$$ the following relation:
\begin{subequations}\label{algebre2:galerkin}
\begin{equation}\label{algebre2}
\begin{split}
0&=  -\oint_\Omega \nabla \bv^h \cdot \bbf(\bu^h) \; d\bx  +\oint_{\partial \Omega}\bv^h \big (\hat{\bbf}_\bn(\bu^h,\bu_b)-\bbf(\bu^h)\cdot \bn\big ) \; d\gamma\\
&\qquad +\sum\limits_{e\in \mathcal{E}_h} \oint_e[\bv^h]\hbbf_\bn(\bu^h_\Kp,\bu^{h}_\Km)\; d\gamma+\sum\limits_{K\subset \Omega}\frac{1}{\#K}\bigg ( \sum\limits_{\sigma,\sigma'\in K} (\bv_\sigma-\bv_{\sigma'})\bigg ( \Phi_\sigma^K(\bu^h)-\Phi_\sigma^{K, Gal}(\bu^h) \bigg ) \bigg )\\
&\qquad \qquad
+\sum\limits_{\Gamma\subset \partial \Omega} \frac{1}{\#\Gamma}\bigg ( \sum\limits_{\sigma,\sigma'\in \Gamma} (\bv_\sigma-\bv_{\sigma'})(\Phi_\sigma^\Gamma \big (\bu^h,\bu_b)-\Phi_\sigma^{Gal,\Gamma}(\bu^h,\bu_b)\big )\bigg )
\end{split}
\end{equation}
where
\begin{equation}\label{galerkin}
\begin{split}\Phi_\sigma^{K, Gal}(\bu^h)&=-\oint_K\nabla\varphi_\sigma\cdot \bbf(\bu^h) \; d\bx +\oint_{\partial K} \varphi_\sigma \hbbf_\bn(\bu^h_\Kp,\bu^{h}_\Km) \; d\gamma,\\
 \Phi_\sigma^{\Gamma, Gal}(\bu^h,\bu_b)&=\oint_\Gamma \varphi_\sigma \big (\hat{\bbf}_\bn(\bu^h,\bu_b)-\bbf(\bu^h)\cdot \bn\big ) \; d\gamma.
\end{split}
\end{equation}
\end{subequations}
Note that 
$$\sum_{\sigma\in K} \Phi_\sigma^{K, Gal}(\bu^h)=\oint_{\partial K} \hbbf_\bn(\bu^h_\Kp,\bu^{h}_\Km) \; d\gamma$$
and
$$\sum_{\sigma\in \Gamma} \Phi_\sigma^{\Gamma, Gal}(\bu^h,\bu_b)=\oint \big (\hat{\bbf}_\bn(\bu^h,\bu_b)-\bbf(\bu^h)\cdot \bn\big ) \; d\gamma.$$
The proof of \eqref{algebre2} is given in appendix \ref{proof}.
The relation \eqref{algebre2} is instrumental in proving the following results.
The first one is proved in  \cite{AbgrallRoe}, and is a generalisation of the classical Lax-Wendroff theorem.
\begin{theorem}\label{th:LW}
Assume the family of meshes $\mathcal{T}=(\mathcal{T}_h)$
is shape regular. We assume that the residuals $\{\Phi_\sigma^{\mathcal{K}}\}_{\sigma\in \KK}$,
 for $\mathcal{K}$ an element or a boundary element of $\mathcal{T}_h$, satisfy: \begin{itemize}
\item For any $M\in \R^+$, there exists a constant $C$ which depends only on the family of meshes $\mathcal{T}_h$ and $M$ such that
 for any $\bu^h\in V^h$ with $||\bu^h||_{\infty}\leq M$, then
$$\big|\Phi^\KK_\sigma({\bu^h}_{|\KK})\big |\leq C\sum_{\sigma, \sigma'\in \KK}|\bu_\sigma^h-\bu_{\sigma'}^h|$$
\item The conservation relations \eqref{conservation:K} and \eqref{conservation:Gamma} with quadrature formulas satisfying assumption \ref{assump:quad}.
\end{itemize}
Then if there exists a constant $C_{max}$ such that the solutions of the scheme \eqref{RD:scheme} satisfy $||\bu^h||_{\infty}\leq C_{max}$ and a function $\bv\in L^2(\Omega)$ such that $(\bu^h)_{h}$ or at least a sub-sequence converges to $\bv$ in $L^2(\Omega)$, then $\bv$ is a weak solution of \eqref{eq1}
\end{theorem}
\begin{proof}
The proof can be found in \cite{AbgrallRoe}, it uses \eqref{algebre2} and some adaptation of the ideas of \cite{kroner}. One of the key arguments comes from the consistency of the flux $\hbbf$ as well as \eqref{quadrature} applied to $\hbbf$\end{proof}

Another consequence of \eqref{algebre2} is the following result on entropy inequalities:
\begin{proposition}\label{th:entropy}
Let $(U,\mathbf{g})$ be a  entropy-flux \textcolor{blue}{pair} for \eqref{eq1} and $\hat{\mathbf{g}}_\bn$ be a numerical entropy flux consistent with $\mathbf{g}\cdot \bn$. Assume that the residuals satisfy:
for any element $K$,
\begin{subequations}\label{entropy}
\begin{equation}\label{entropy:1}
\sum_{\sigma \in K}\langle\bV_\sigma, \Phi_\sigma^K\rangle \geq \oint_{\partial K} \hat{\mathbf{g}}_\bn(\bV^h_\Kp,\bV^{h}_\Km) \; d\gamma
\end{equation}
and for any boundary edge $e$,
\begin{equation}\label{entropy:2}
\sum_{\sigma \in e}\langle\bV_\sigma ,  \Phi_\sigma^e\rangle  \geq \oint_{e} \big (\hat{\mathbf{g}}_\bn(\bV^h_\Gamma,\bV_b)- \mathbf{g}(\bV^h)\cdot \bn \big )\; d\gamma.
\end{equation}
\end{subequations}
Then, under the assumptions of theorem \ref{th:LW}, the limit weak solution also satisfies the following entropy inequality: for any $\varphi\in C^1(\overline{\Omega})$, $\varphi\geq 0$, 
$$-\oint_\Omega \nabla \varphi\cdot \mathbf{g}(\bV) \; d\bx+\oint_{\partial\Omega^-}\varphi\;\mathbf{g}(\bV_b)\cdot \bn  \; d\gamma \leq 0.$$
\end{proposition}
\begin{proof}
The proof is similar to that of theorem \ref{th:LW} once we are writing the conservative variables in term of the entropy variable
\end{proof}

\bigskip
Another consequence of \eqref{algebre2} is the following condition under which  one can guarantee to have a $k+1$-th order accurate scheme. We first introduce the (weak) truncation error
\begin{equation}\label{truncation}
\mathcal{E}\bigl(\bu^h, \varphi\bigr) = 
\sum_{\sigma \in \mathcal{S}_h}\varphi_\sigma\bigg [
\sum_{K\subset \Omega, \sigma\in K} \Phi_\sigma^K+\sum_{\Gamma\subset \partial\Omega, \sigma\in \Gamma}\Phi_\sigma^\Gamma\bigg ].
\end{equation}

\begin{proposition}\label{prop:approx}
Assuming that:
\begin{enumerate}
\item  The solution of the \emph{steady problem} $\bu$ is smooth enough,
\item The approximation $\pi_h(\bu)$ of $\bu$ is accurate with  order 
$k+1$, and $\bbf$ Lipschitz continuous,
\item The quadrature formula \eqref{quad:K} and \eqref{quad:f} are of order $k+1$: for any $\psi$ of class at least $C^{k+1}$, for any element $K$ and face $f$,
$$\int_K \psi(\bx) \;d\bx =\oint_K \psi(\bx) \; d\bx+O(h^{d+k+1}), \qquad \int_f \psi \;d\gamma=\oint_f \psi \; d \gamma +O(h^{d+k}),$$
\item The residuals,
computed with the interpolant $\pi_h(\bu)$ of the solution, are such that for any element $K$ and boundary element $\Gamma$
\begin{equation}\label{eq: residual accuracy}
\Phi_\sigma^K(\pi_h(\bu)) = \mathcal{O}\bigl(h^{k+d}\bigr), \qquad 
\Phi_\sigma^\Gamma(\pi_h(\bu))=\mathcal{O}\bigl(h^{k+d-1}\bigr)
\end{equation}
\end{enumerate}
then the truncation error satisfies the following relation
\begin{equation*}
|\mathcal{E}\bigl(\pi_h(\bu), \varphi\bigr)| \le C( \bbf, \bu)\; ||\varphi||_{H^1(\Omega)}\; h^{k+1},
\end{equation*}
with $C$ a constant which depends only on $\bbf$, and $||\bu||_\infty$.
\end{proposition}
The proof can be found in appendix \ref{proof}.
%
}

{
\paragraph{Local conservation.} In \cite{abgrall:hal-01573592}, we show that schemes satisfying the conservation requirements \eqref{conservation:K} and \eqref{conservation:Gamma}  are indeed locally conservative. More precisely, for each degree of freedom, one can define a control volume $C\sigma$ of polygonal type, and numerical flux $\hat{f}_{\bn_{\sigma\sigma'}}(\bu^h)$ that depend on the values $\bu_{\sigma'}$ in  a neighborhood of $C_\sigma$ and the local normal of $C_\sigma\cup C_{\sigma'}$ which is planar. This set of connection defines a dual  graph. By construction, we have 
$$\hat{f}_{\bn_{\sigma\sigma'}}(\bu^n)=-\hat{f}_{\bn_{\sigma'\sigma}}(\bu^h),$$
which ensure local conservation. Details can be found in \cite{abgrall:hal-01573592}, but we give two examples in appendix \ref{control}.

This means that \eqref{RD:scheme} can be rewritten as
$$\sum_{\sigma'\in \mathcal{N}_\sigma}\hat{f}_{\bn_{\sigma\sigma'}}(\bu^h)=0$$
where $\mathcal{N}_\sigma$ is the set of degrees of freedom connected to $\sigma$ by the dual graph defined from the cells.
}

\subsection{Construction of entropy conservative schemes}
\red{ In the rest of the paper, $\bV^h\in V^h$ represents the approximation of the entropy variable $\bV=\nabla_\bbu U(\bbu)$ in $V^h$, and since, the mapping $\bu\rightarrow \bV(\bu)$ is one-to-one, $\bu^h:=\bu(\bV^h)$ here. Hence in general it is not a polynomial.}

Starting from a scheme $\{\Phi_\sigma\}$, we show in this section how to construct a new scheme, with residuals $\{\Phi_\sigma'\}$ such that
\begin{equation}\label{ent:1}
\sum\limits_{\sigma\in K} \langle \bV_\sigma, \Phi'_\sigma\rangle =\oint_{\partial K}\hbbg_\bn(\bV^h_\Kp, \red{\bV_\Km^{h}})\; d\gamma.\end{equation}
\red{This relation is motivated by Proposition \ref{th:entropy} where the inequality is replaced by an equality, so that combined with Theorem \ref{th:LW}, we know  that if the scheme converges, the limit solution will satisfy two conservation relations: one on $\bu$ and one on the entropy. \emph{The equality, together with the remark of the appendix \ref{control} shows that we get indeed local conservation of the entropy, too.}}

Following \cite{svetlana,paola},the idea is to write:
\begin{equation}
\label{entropy:perturb}
\Phi'_\sigma=\Phi_\sigma+\mathbf{r}_\sigma
\end{equation} such that conservation is still guaranteed
$$\sum\limits_{\sigma\in K} \big (\Phi_\sigma^K+\mathbf{r}_\sigma\big )=\int_{\partial K}\hbbf_\bn(\bV^h_\Kp, \red{\bV_\Km^{h}})\; d\gamma,$$
i.e.
\begin{subequations}\label{entropy:rel}
\begin{equation}\label{ent:2bis}
\sum_{\sigma\in K}\mathbf{r}_\sigma=0,
\end{equation}
and \eqref{ent:1} holds true, i.e.
\begin{equation}\label{ent:1bis}
\sum\limits_{\sigma\in K} \langle \bV_\sigma, \mathbf{r}_\sigma\rangle =\oint_{\partial K}\hbbg_\bn(\bV^h_\Kp, \red{\bV^{h}_\Km})\; d\gamma-\sum\limits_{\sigma\in K}\langle \bV_\sigma, \red{\Phi_\sigma}\rangle:=\mathcal{E}\end{equation}
\end{subequations}
\red{The two relations \eqref{entropy:rel} defines a linear system  with always at least two unknowns: the solution which is always valid is}
\begin{subequations}
\label{ent}
\begin{equation}
\label{ent1}
\mathbf{r}_\sigma=\red{\alpha} \big ( \bV_\sigma-\bar\bV\big ) \text{ with } \bar\bV =\frac{1}{\#K}\sum_{\sigma\in K}\bV_\sigma.\end{equation}
Since 
\begin{equation*}
\begin{split}
\sum\limits_{\sigma\in K} \langle \bV_\sigma, \bV_\sigma-\bar\bV\rangle&=\sum\limits_{\sigma\in K} \langle \bV_\sigma-\bar\bV, \bV_\sigma-\bar\bV\rangle+\langle \bar\bV, \sum\limits_{\sigma\in K} (\bV_\sigma-\bar\bV)\rangle=\sum\limits_{\sigma\in K} \langle \bV_\sigma-\bar\bV, \bV_\sigma-\bar\bV\rangle\\
&=\sum_{\sigma\in K} (\bV_\sigma-\bar\bV)^2
\end{split}
\end{equation*}
since $\sum\limits_{\sigma\in K} (\bV_\sigma-\bar\bV)=0$ by construction. Thus we get
\begin{equation}
\label{ent2}\red{\alpha}=\dfrac{\mathcal{E}}{\sum\limits_{\sigma\in K}(\bV_\sigma-\bar\bV)^2}.\end{equation}
\end{subequations}
{\begin{remark}
\begin{enumerate}
\item 
Since \eqref{ent:1bis} means that the new set of residuals satisfy a conservation relation for the entropy on each element, from \cite{abgrall:hal-01573592}
the new scheme will be locally conservative for the variable $\bu$ and the entropy.
\item Relation \eqref{ent2} becomes singular  for a constant state. In practice, we use
\begin{equation}
\label{ent2:2}\red{\alpha}=\dfrac{\mathcal{E}}{\sum\limits_{\sigma\in K}(\bV_\sigma-\bar\bV)^2+\varepsilon}\end{equation}
and we have taken $\varepsilon=10^{-20}$.
\item Interpretation of \eqref{ent}.
In fact adding $\alpha(\bV_\sigma-\bar \bV)=\frac{\alpha}{\#K}\sum\limits_{\sigma'\in K} (\bV_\sigma-\bV_{\sigma'})$ to each residual is like adding diffusion. If we had a Cartesian mesh, this amounts to adding an approximation of $$\Delta x\dpar{}{x}\big ( \alpha \dpar{\bV}{x}\big ) +\Delta y\dpar{}{y}\big ( \alpha \dpar{\bV}{y}\big ).$$ Since $\alpha $ has a priori no sign, depending on the local entropy production of the original scheme, one adds or removes the right amount to be entropy conservative.
\end{enumerate}
\end{remark}
}
Now, let \red{us} have a look at the accuracy of such a scheme. Let us remind we are looking at \emph{steady} problems. According to the accuracy conditions \eqref{eq: residual accuracy}, we assume  $\Phi_\sigma=O(h^{k+d})$. Similarly, since the numerical flux $\hbbg$ is Lipschitz continuous and consistent, { we also have \text{point-wise}, for a smooth steady \red{exact} solution $\bV$, $\text{ div }\bbg(\bV)=0$, so that
$$0=\int_K \text{ div }\bbg(\bV)\; d\bx=\int_{\partial K} \bbg(\bV)\cdot\bn\; d\gamma$$ and then}
\begin{equation}\label{calcul}
\begin{split}
\oint_{\partial K}\hbbg_\bn(\bV^h_\Kp, \red{\bV_\Km^{h}})\; d\gamma&=\oint_{\partial K}\hbbg_\bn(\bV^h_\Kp, \red{\bV_\Km^{h}})\; d\gamma-\int_{\partial K}\bbg(\bV)\cdot\bn\; d\gamma \\
& =\oint_{\partial K}\big ( \hbbg_\bn(\bV^h_\Kp, \red{\bV_\Km^{h}})-\bbg(\bV)\cdot\bn\big ) \; d\gamma+\oint_{\partial K}\bbg(\bV)\cdot\bn\; d\gamma-\int_{\partial K}\bbg(\bV)\cdot\bn\; d\gamma\\
&=O(h^{k+d})
\end{split}
\end{equation}
provided the quadrature formula on the boundary is \red{at least }of order $k$. This is the case by assumption. { In general, we cannot have more than $O(h^{k+d})$ because even if the integration is exact, we cannot have
$$\oint_{\partial K}\big ( \hbbg_\bn(\bV^h_\Kp, \bV_\Km^{h})-\bbg(\bV)\big ) \; d\gamma=O(h^{k+d+1}).$$}
 Hence, $\mathcal{E}=O(h^{k+d})$. However, for smooth problems, we have
$$\red{\alpha}=O(h^{k-2+d}),$$ so that the correction $\mathbf{r}_\sigma$ is only $O(h^{k-1+d})$  a priori: we loose one order. Thus we get
{ \begin{proposition}
The scheme defined by \eqref{entropy:perturb}, \eqref{ent} is localy entropy conservative and at least of order $k$ at steady state.
\end{proposition}
}

Up to now, we have shown, starting form a general additional conservation constraint,  how to modify a scheme so that the modified scheme will satisfy one additional conservation relation, and what we see is that a priori we loose one order of accuracy. This is a general statement that assumes there is no particular connection between the 'old' conservation relation and the new one. However, there is a connection for the entropy flux, it is given by the following relation, see \cite{TadmorActa}: the entropy flux $g$ is related to the flux $\bbf$ by
the relation
$$\red{\bbg(\bV)=\langle \bV, \bbf(\bV)\rangle -\theta(\bV)}$$
 In addition, the potential satisfies
$$\red{\nabla_\bV \theta(\bV)=\bbf(\bV).}$$
Using this, and a proper definition of the entropy numerical flux, it is possible to do better. We explore this question for the discontinuous Galerkin method, and the non linear RD scheme. The discussion on the discontinuous Galerkin method will also embed the continuous Galerkin approximation.

\subsubsection{Discontinuous Galerkin methods}

In the case of the discontinuous Galerkin methods, we get
$$\mathcal{E}=-\oint\nabla \bV^h\cdot \bbf(\bV^h)\; d\bx+\oint_{\partial K}\langle \bV^h, \bbf_\bn(\bV^h_\Kp,\bV_\Km^{h})\rangle\; d\gamma-\oint_{\partial K} \hbbg_\bn(\bV^h_\Kp,\bV_\Km^{h})\; d\gamma.$$
The precise definition of $\hbbg_\bn$ will be done later in this section. Since 
\red{point-wise\footnote{\red{Note that we are computing the divergence (in $\bx$) of the function $ \langle\bV^h,\bbf(\bV^h)\rangle $}}}, we have $$ \nabla\bV^h\cdot \bbf(\bV^h)=\text{ div }(\langle\bV^h,\bbf(\bV^h)\rangle )-\bV^h\cdot \nabla \bbf(\bV^h)$$ and
\red{$\bV^h\cdot \text{ div }\bbf(\bV^h)=\text{ div }\bbg(\bV^h)$}, we can write:
\begin{equation}\label{dg:calcul}
\begin{split}
\mathcal{E}&=-\oint_K\nabla\bV^h\cdot \bbf(\bV^h)\; d\bx+\oint_{\partial K}\langle\bV^h,\hbbf_\bn(\bV^h_\Kp,\bV_\Km^{h})\; d\gamma-\oint_{\partial K} \hbbg_\bn(\bV^h_\Kp,\bV_\Km^{h})\; d\gamma\\
&=-\oint_K \text{ div }\big (\langle \bV^h, \bbf(\bV^h)\rangle \big ) \; d\bx+\oint_{K}\bV^h\cdot \text{div }\bbf(\bV^h)\; d\bx+
\oint_{\partial K}\langle\bV^h,\hbbf_\bn(\bV^h_\Kp,\bV_\Km^{h})\; d\gamma-\oint_{\partial K} \hbbg_\bn(\bV^h_\Kp,\bV_\Km^{h})\; d\gamma\\
&=-\oint_K \text{ div }\big (\langle \bV^h, \bbf(\bV^h)\rangle \big ) \; d\bx+\oint_{K} \text{div }\bbg(\bV^h)\; d\bx+
\oint_{\partial K}\langle\bV^h,\hbbf_\bn(\bV^h_\Kp,\bV_\Km^{h})\; d\gamma-\oint_{\partial K} \hbbg_\bn(\bV^h_\Kp,\bV_\Km^{h})\; d\gamma\\
&=\underbrace{-\bigg (\oint_K \text{ div }(\langle \bV^h,\bbf(\bV^h)\rangle )\; d\bx+\int_K \text{ div }(\langle \bV^h,\bbf(\bV^h)\rangle )\; d\bx\bigg )}_{\text{(I)}}
+\underbrace{\bigg ( \oint_K\text{ div }\bbg(\bV^h)\; d\bx-\int_K\text{ div }\bbg(\bV^h)\; d\bx\bigg )}_{\text{(II)}}\\
& \qquad \qquad\qquad -\int_K \big ( \text{ div }(\langle \bV^h, \bbf(\bV^h)\rangle-\bbg(\bV^h) \big )\; d\bx
+\oint_{\partial K} \big ( \langle\bV^h,\hbbf_\bn(\bV^h_\Kp,\bV_\Km^{h})\; d\gamma-\hbbg_\bn(\bV^h_\Kp,\bV_\Km^{h}) \big )\; d\gamma\\
& =(I)+(II)+\underbrace{\bigg ( \int_{\partial K} \big ( \bV^h \bbf_\bn(\bV^h)-\bbg_\bn(\bV^h) \big ) \; d\gamma- \oint_{\partial K} \big ( \bV^h \cdot\bbf_\bn(\bV^h)-\bbg_\bn(\bV^h) \big ) \; d\gamma\bigg )}_{\text{III}} \\
&\qquad \qquad + 
\oint_{\partial K}\bigg ( \langle \bV^h \cdot\big ( \hbbf_\bn(\bV^h_\Kp,\bV_\Km^{h})- \bbf_\bn(\bV^h)\big )\rangle-\big ( \hbbg_\bn(\bV^h_\Kp,\bV_\Km^{h}) -\bbg_\bn(\bV^h) \big )\bigg )\; d\gamma
\end{split}
\end{equation}
Since $\bV^h$ is a polynomial, if the quadrature formulas are of order $k$, we get that $I$, $II$ and $III$ in \eqref{dg:calcul} are 
$$I=O(h^{d+k+1}), \quad II=O(h^{d+k}), \quad III=O(h^{d+k}).$$

Let us have a look at the last term. If we define the entropy flux to be
\begin{equation}
\label{entropyflux:DG}
\hbbg_\bn(\bV^h,\bV^{h,-})=\langle \{\bV^h\}, \hbbf_\bn(\bV_|Kp^h,\bV_\Km^{h})-\theta(\{\bV^h\}),
\end{equation}
we have
\begin{equation*}
\begin{split}
-\langle\bV^h,\bbf(\bV^h)\rangle&+\langle\bV^h,\hbbf_\bn(\bV_\Kp^h,\bV_\Km^{h})+\bbg(\bV^h)\cdot\bn-\hbbg_\bn(\bV_{\Kp}^h,\bV_{\Km}^{h})\\
&=\langle \bV^h-\{\bV^h\},\hbbf_\bn(\bV_\Km^h,\bV_\Km^{h})\rangle+\big (\theta(\{\bV^h\})-\theta(\bV^h)\big )\cdot \bn\\
&=\langle \bV^h-\{\bV^h\},\hbbf_\bn(\bV_\Kp^h,\bV_\Km^{h})\rangle-\langle \nabla_\bv\theta(\bar \bV), \bV^h-\{\bV^h\}\rangle\\
&=\langle \bV^h-\{\bV^h\},\hbbf_\bn(\bV_\Kp^h,\bV_\Km^{h})-\nabla_\bV\theta(\bar \bV)\rangle
\end{split}
\end{equation*}
where $\bar \bV=t \{\bV^h\}+(1-t)\bV$ by abuse of language
\footnote{We write $\theta(\{\bV^h\})-\theta(\bV^h)$ component by component. For the component $\# l$, we get
$$\theta_l(\{\bV^h\})-\theta_l(\bV^h)=\langle \dpar{\theta_l}{\bV}(t_l \{\bV^h\}+(1-t_l)\bV), \{\bV^h\}-\bV\rangle,$$
 $t_l\in [0,1]$  Hence
 $$\langle \bV^h-\{\bV^h\},\hbbf_\bn(\bV_\Kp^h,\bV_\Km^{h})-\nabla_\bV\theta(\bar \bV)\rangle$$
 means, by abuse of language
 $$\sum\limits_{l=1}^p \big (\bV^h-\{\bV^h\}\big )_l\bigg ( \big (\hbbf_\bn(\bV_\Kp^h,\bV_\Km^{h})\big )_l-\dpar{\theta_l}{\bV}(t_l \{\bV^h\}+(1-t_l)\bV)
 \bigg ).$$
 }.
 Since $\nabla_\bV\theta=\bbf$, we have indeed
\begin{equation*}
\begin{split}
-\langle\bV^h,\bbf(\bV^h)\rangle&+\langle\bV^h,\hbbf_\bn(\bV_\Kp^h,\bV_\Km^{h})+\bbg(\bV^h)\cdot\bn-\hbbg_\bn(\bV_\Kp^h,\bV_\Km^{h})
\\&=\langle \bV^h-\{\bV^h\},\hbbf_\bn(\bV_\Kp^h,\bV_\Km^{h})-\bbf(\bar \bV)\rangle=O(h^{2(k+1)})
\end{split}
\end{equation*}
because the numerical flux is Lipschitz continuous. Hence 
$$\oint_{\partial K} \bigg ( -\langle\bV^h,\bbf(\bV^h)\rangle+\langle\bV^h,\hbbf_\bn(\bV_\Kp^h,\bV_\Km^{h})+\bbg(\bV^h)\cdot\bn-\hbbg_\bn(\bV_\Km^h,\bV_\Km^{h})\bigg ) d\gamma=O(h^{d+2k+1})$$
provided suitable quadrature formulas.
In the end we get
\begin{proposition}
{ The scheme defined by \eqref{entropy:perturb}, \eqref{ent} is localy entropy conservative. In addition  
 $\mathcal{E}=\sum\limits_{\sigma\in K}\langle \bV_\sigma, \Phi_\sigma\rangle-\oint_{\partial K}\hbbg_\bn(\bV_\Kp^h, \bV_\Km^{h})\; d\gamma$, with
$$\hbbg_\bn(\bV_\Kp^h, \bV_\Km^{h})=\langle \{\bV\},\hbbf\rangle-\theta(\{\bV^h\})$$
satisfies
\begin{equation}
\label{error:DG}
\mathcal{E}=O(h^{d+k+1})
\end{equation}
if the quadrature formula are of order $k$ for the surface/volume integrals and $k+1$ for the boundary, so that the scheme \eqref{entropy:perturb}-eqref{ent} is stil of order $k+1$ at steady state.}
\end{proposition}
\begin{remark}
\begin{enumerate}
\item We do not use the fact we are solving a steady problem: the same would be true for the semi-discrete unsteady problem.
\item We have only used approximation properties: the same would be true when non linear limiting is applied, provided the formal order is kept for sommth solutions
\end{enumerate}
\end{remark}
\subsubsection{Non linear RD schemes}
Let us recall that for the non linear RD schemes defined in appendix \ref{RDS}, the accuracy condition is that the boundary integrals are of order $k-1$, since then
$$\oint_{\partial K}\hbbf_\bn(\bV_\Kp^h, \bV_\Km^{h})\; d\gamma=O(h^{k+d}).$$ If $\hbbg_\bn$ is a Lipschitz continuous flux that is consistent with the entropy flux, we also have
$$\oint_{\partial K}\hbbg_\bn(\bV_\Kp^h, \bV_\Km^{h})\; d\gamma=O(h^{k+d}).$$
Let us look at $\mathcal{E}$ in more detail when $\Phi_\sigma=\mathbf{\beta}_\sigma\Phi$, $\sum\limits_{\sigma\in K}\mathbf{\beta}_\sigma=\text{Id}$. 
We have, defining again $\hbbg=\langle \{\bV\},\hbbf\rangle-\theta(\{\bV^h\})$ and $\BV=\sum\limits_{\sigma\in K}\mathbf{\beta}_\sigma \bV_\sigma$, and taking into account we assume a smooth steady solution
\begin{equation*}
\begin{split}
\mathcal{E}&=\oint_{\partial K} \hbbg_n(\bV_\Kp^h, \bV_\Km^{h})\; d\gamma-\oint_{\partial K} \langle\BV, \hbbf_\bn(\bV_\Km^h, \bV_\Km^{h})\; d\gamma\\
&=\oint_{\partial K} \bigg (\hbbg_n(\bV_\Kp^h, \bV_\Km^{h})-\bbg(\bV)\bigg )\; d\gamma-\oint_{\partial K}\bigg (  \langle\BV, \hbbf_\bn(\bV_\Kp^h, \bV_\Km^{h})-\langle \BV, \bbf(\bV)\rangle \bigg )\; d\gamma.
\end{split}
\end{equation*}
We look at each term.
First, using Taylor formula, and the fact that $\nabla_\bV \theta(\bV)=\bbf(\bV)$, 
\begin{equation*}
\begin{split}
\hbbg_n(\bV_\Kp^h, \bV_\Km^{h})-\bbg(\bV)&=\langle\{\bV^h\},\hbbf_\bn(\bV_\Kp^h, \bV_\Km^{h})\rangle-\theta(\{\bV^h\}-\langle \bV, \bbf(\bV)\rangle +\theta(\bV)\\
&=\langle\{\bV^h\}-\bV, \bbf(\bV)\rangle+\langle\{\bV^h\}, \hbbf_\bn(\bV_\Kp^h, \bV_\Km^{h})\rangle-\bbf(\bV)\rangle\\
&\qquad +\langle \nabla_\bV \theta(\bV), \bV-\{\bV^h\}\rangle+O\big ((\bV-\{\bV^h\})^2\big )\\
&=\langle\{\bV^h\}, \hbbf_\bn(\bV_\Kp^h, \bV_\Km^{h})-\bbf(\bV)\rangle+O\big ((\bV-\{\bV^h\})^2\big )
\end{split}
\end{equation*}
so that
\begin{equation*}
\begin{split}
\hbbg_\bn(\bV_\Kp^h, \bV_\Km^{h})-\bbg(\bV) -\langle \bV, \hbbf_\bn(\bV_\Km^h, \bV_\Km^{h})-\bbf(\bV)\rangle&=\langle \{\bV^h\}-\bV, \hbbf_\bn(\bV_\Kp^h, \bV_\Km^{h})-\bbf(\bV)\rangle +O\big ((\bV-\{\bV^h\})^2\big )\\
&=O(h)\times O(h^{k+1})
\end{split}
\end{equation*}
if the boundary quadrature formula is of order $k$.
We have
{\begin{proposition}
The scheme defined by \eqref{entropy:perturb}, \eqref{ent} is localy entropy conservative. In addition, $\mathcal{E}=\sum\limits_{\sigma\in K}\langle \bV_\sigma, \Phi_\sigma\rangle-\oint_{\partial K}\hbbg_\bn(\bV_\Kp^h, \bV_\Km^{h})\; d\gamma$, with
$$\hbbg_\bn(\bV_\Kp^h, \bV_\Km^{h})=\langle \{\bV\},\hbbf\rangle-\theta(\{\bV^h\})$$
satisfies 
\begin{equation}
\label{error:RD}
\mathcal{E}=O(h^{k+d+1})
\end{equation}if  the boundary quadrature formula is of order $k$.
\end{proposition}}

\begin{remark}
Here we must use the fact we are solving a steady problem.
\end{remark}

\subsection{Construction of entropy stable schemes}\label{sec:construction}
Starting form the residuals $\{\Phi^K_\sigma\}_{K\in \mathcal{T}_h, \sigma\in K}$, we want to construct a scheme $\{\Phi^{'K}_\sigma\}_{K\in \mathcal{T}_h, \sigma\in K}$, that is entropy stable, i.e., for any $K$,
$$\sum\limits_{\sigma\in K}\langle\bV_\sigma,\Phi^{' K}_\sigma (\bV^h) \rangle\geq \oint_{\partial K}\hbbg_\bn(\bV_\Kp^h,\bV_\Km^{h})\; d\gamma.$$
Starting from the previous construction, and dropping the superscript $K$ because we will consider only quantities in $K$, the idea is to write the new residuals as:
\begin{equation}
\label{RD_Stable}
\Phi'_\sigma(\bV^h)=\Phi_\sigma+\mathbf{r}_\sigma+\Psi_\sigma
\end{equation}
where the $\mathbf{r_\sigma}$ are defined by \eqref{ent} and the $\Psi_\sigma$ satisfy
$$\sum\limits_{\sigma\in K}\Psi_\sigma=0 \text{ and } \sum\limits_{\sigma\in K}\langle\bV_\sigma,\Psi_\sigma(\bV^h)\rangle\geq 0$$
without violating the accuracy condition \eqref{eq: residual accuracy}.

Let us consider these two expressions for $\Psi_\sigma(\bV^h)$ which obviously satisfy the conservation requirement
$$\sum\limits_{\sigma\in K} \Psi_\sigma(\bV^h)=0$$
since $\sum\limits_{\sigma\in K}\varphi_\sigma=1$:
\begin{enumerate}
\item With jumps: We define 
\begin{equation}
\label{entropy:sol1}
\Psi_\sigma=\theta h_K^2\oint_{\partial K} [\nabla_\bx\varphi_\sigma]\cdot [\nabla_\bx \bV^h]\; d\gamma
\end{equation}
We have 
$$\sum\limits_{\sigma\in K} \langle\bV_\sigma,\Psi_\sigma(\bV^h)\rangle=\theta\; h_K^2\sum_{e \text{ edge of K}}\oint_K[\nabla_\bx \bV^h]^2\; d\gamma\geq 0$$
for any $\theta>0$. It is also easy to check that if $\bV^h$ is an approximation of order $k$ of a smooth $\bV$, we have, \emph{provided the quadrature formula is of order $k$} that
$$\Psi_\sigma=O(h^{k+d})$$
so that the accuracy requirement are met.
\item With streamline: we define
\begin{equation}
\label{entropy:sol2}
\Psi_\sigma=\theta h_K\oint_K \big ( \nabla_\bV\bbf(\bV^h)\cdot \nabla_\bx \varphi_\sigma\big ) \; \tau_K \big (\nabla_\bV\bbf(\bV^h)\cdot \nabla_\bx \bV^h \big ) \; d\bx.
\end{equation}
Clearly,
$$\sum\limits_{\sigma\in K} \langle\bV_\sigma,\Psi_\sigma(\bV^h)\rangle=\theta h_K \oint_K \big ( \nabla_\bV\bbf(\bV^h)\cdot \nabla_\bx \bV_\sigma\big ) \; \tau_K \big (\nabla_\bV\bbf(\bV^h)\cdot \nabla_\bx \bV^h \big ) \; d\bx\geq  0
$$
for any $\theta>0$ and any $\tau>0$
In practice, we take $\tau_K$ as the scalar/matrix defined by 
$$\tau^{-1}=\sum_{\sigma'\in K} K_{\sigma'}^-, K_\sigma=\dfrac{1}{|K|}\int_K \overline{\nabla_\bu\Bf}\cdot\nabla_\bx \varphi_\sigma\; d\bx, \qquad K_\sigma^+=\max(K_\sigma,0), K_\sigma^-=\min(K_\sigma,0).$$
It is also easy to check that if $\bV^h$ is an approximation of order $k$ of a smooth $\bV$, we have, \emph{provided the quadrature formula is of order $k$} that
$$\Psi_\sigma=O(h^{k+d})$$
so that the accuracy requirement are met.
\end{enumerate}

In practice, the quadrature formula involved in the edge integrals and the surface integrals need to be such that:
\begin{itemize}
\item For $\oint_e [\nabla_\bx \bV^h]^2\; d\gamma$: if $\bV^h$ is not constant, 
$$\oint_e [\nabla\bV^h]^2\; d\gamma>0.$$
\item For $\oint_K \big ( \nabla_\bV\bbf(\bV^h)\cdot \nabla_\bx \bV^h\big ) \; \tau_K \big (\nabla_\bV\bbf(\bV^h)\cdot \nabla_\bx \bV^h \big ) \; d\bx$: if 
$\nabla_\bV\bbf(\bV^h)\cdot \nabla_\bx \bV^h$ is not identicaly $0$, then
$$\oint_K \big ( \nabla_\bV\bbf(\bV^h)\cdot \nabla_\bx \bV_\sigma\big ) \; \tau_K \big (\nabla_\bV\bbf(\bV^h)\cdot \nabla_\bx \bV^h \big ) \; d\bx>0.$$
In \cite{abgrallLarat}, we have provided minimal conditions on these quadrature formula, such that this conditions are met. \red{Note that these quadrature formulas may be non-consistent: the only things we need is to keep the accuracy, and have a strictly positive entropy production.}
\end{itemize}

\section{Numerical illustrations}{
\label{numerical:solutions}
In this section, we illustrate the behaviour of the methods on  non linear examples.  We first study the conservation issues, for mass and entropy, and then show the behavior of the scheme on a Burger-like problem.
\subsection{Conservation}
In $\Omega=[-20,20]\times[-20,20]$ we look at the problem
\begin{subequations}
\label{sqrt}
\begin{equation}
\label{sqrt.1}\dpar{u}{t}+\dpar{\sqrt{u}}{x}+\dpar{u}{y}=0
\end{equation}
with 
\begin{equation}
\begin{split}
\label{sqrt.2}u_0(x,y)&=\left\{ \begin{array}{cc}
1 & \text{ if } r=\sqrt{(x+5)^2+(y+5)^2}>0.5\\
1+2\sin\big ( \frac{\pi}{4}(1-2r)\big ) &\text{ else.}
\end{array}\right .
\\
u(x,y,t)&=u_0(x,y) \text{ for } t\geq 0, \text{ and } \bx\in \partial\Omega .
\end{split}
\end{equation}
\end{subequations}
 We integrate for $t\in [0,5]$: at time $t=5$, the non constant part of the solution has not reached the boundary, this is why we have taken such a large domain. We want to avoid any interference with the boundary conditions. The flux is non polynomial, so that any standard quadrature rule will \emph{never} be exact.

Following the notations of \eqref{RD:scheme},
we integrate in time with  an Euler forward method
\begin{equation}
\label{euler:forward}
\bu_\sigma^{n+1}=\bu_\sigma^n-\dfrac{\Delta t}{C_\sigma}\bigg ( \sum\limits_{K\subset\Omega, \sigma\in K}\Phi_\sigma^K(\bu^{n,h})+\sum\limits_{\Gamma\subset\partial\Omega, \sigma\in \Gamma}\Phi_\sigma^\Gamma(\bu^{n,h})\bigg ),
\end{equation}
where $C_\sigma$  is a dual control volume. This method is not accurate in time, but we can check numerically the conservation of mass,
$$
M=\sum_{\sigma\in\Omega}|C_\sigma| \bu_\sigma.
$$
An entropy for \eqref{sqrt.1} is $U(\bu)=\frac{\bu^2}{2}$, the entropy flux is $\bbg(\bu)=(\frac{\bu^2}{2}, \frac{1}{3}\bu \sqrt{\bu}).$
Since $$\langle \bV(\bu^n), \bu^{n+1}-\bu^n\rangle\leq U(\bu^{n+1})-U(\bu^{n})$$
we will not get conservation of entropy, but we will check that, the entropy residuals defined by \eqref{entropy:perturb} $\Phi_\sigma'$ are such that
\begin{equation}
\label{entropy_flux}
\sum_{\sigma\in \Omega} \bigg ( \sum\limits_{K\subset\Omega, \sigma\in K}\langle \bV_\sigma,\Phi_\sigma^K(\bu^{n,h})\rangle +\sum\limits_{\Gamma\subset\partial\Omega, \sigma\in \Gamma}\langle \bV_\sigma, \Phi_\sigma^\Gamma(\bu^{n,h})\rangle \bigg ) =0.
\end{equation}
Up to our knowledge, any of  the  so-called entropy conservative schemes show rigorous conservation of the spatial discretisation terms only. To get also entropy conservation at the time-discrete level, it seems that the scheme must be implicit in time, see \cite{hiltebrand14:_entrop_galer}. This is the reason why we are not showing 
$$E=\int_{\Omega} U(\bu^h)\; d\bx$$ but  how well \eqref{entropy_flux} is fullfiled in figure \ref{entropy:conservation}.
\begin{figure}[h]
\begin{center}
\subfigure[pure Galerkin, mass]{\includegraphics[width=0.45\textwidth]{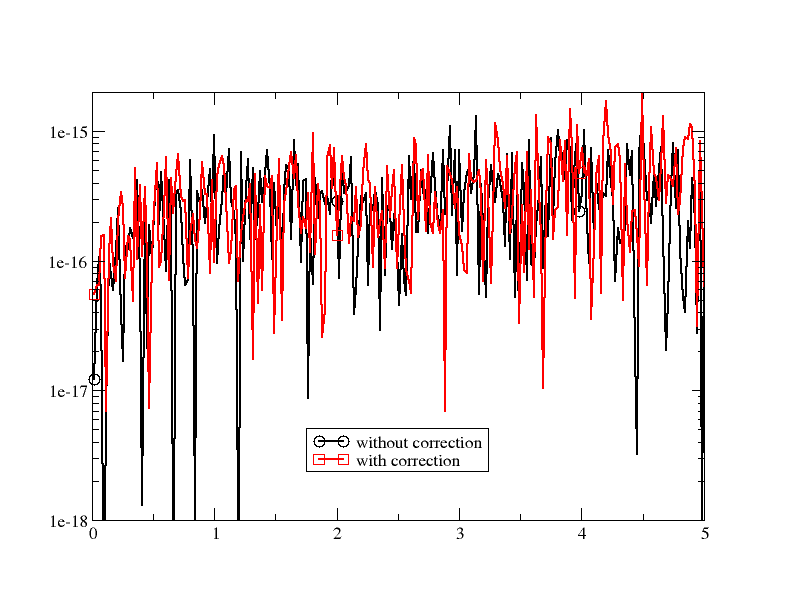} }
\subfigure[pure Galerkin, entropy]{\includegraphics[width=0.45\textwidth]{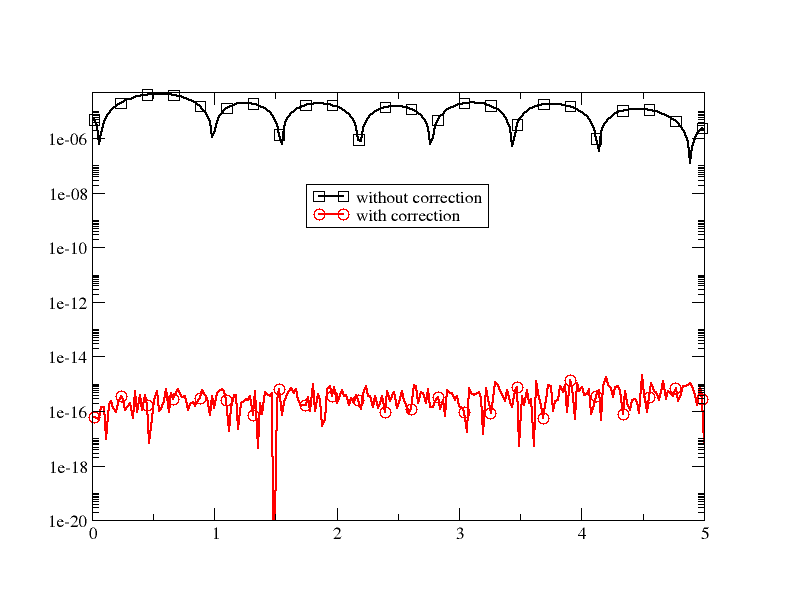}}
\end{center}
\caption{\label{entropy:conservation} Plot of \eqref{entropy_flux} for the pure galerkin scheme, with and without the correction}
\end{figure}
We see that the mass is always conserved up to machine accuracy, while the correction has a clear effect on the spatial conservation of entropy. This was shown for the pure Galerkin scheme.

\subsection{Accuracy}\label{sec:accuracy}
In this section we want to show that the correction does not spoil the accuracy. This is done on the following test case: in $\Omega=[0,1]\times [0,1]$, 
\begin{subequations}\label{expo}
\begin{equation}
\label{expo:1}
\dpar{\sinh u}{x}+\dpar{u}{y}=0 
\end{equation}
with the boundary condition, for $y=0$,
\begin{equation}
\label{expo:2}
u(x,0)=x-\frac{1}{2}
\end{equation}
\end{subequations}
The solution is regular and computed by the method of characteristics. We have chosen the exponential flux to make sure that none of the standard quadrature formula is exact. The quadrature formula are:
\begin{itemize}
\item Linear approximation. the quadrature point in the triangle $K$ is its centroid, and the weight is $1$. On the boundary, we have chosen the Gaussian formula associated to the Legendre polynomials of degree 2: there are two points of weight $\tfrac{1}{2}$ that are the image of $\pm \tfrac{1}{\sqrt{3}}$ in the mappings form $[-1,1]$ to any of the edges of $K$.
\item Quadratic approximation. The surface integrals are computed with $7$ quadrature points, the formula is exact for degree 5
\begin{itemize}
\item The centroid with weight $0.225$,
\item The points defined by the barycentric coordinates $(a,b,b)$,
$$a=0.797426985353087, \qquad a+2b=1$$
and cyclic permutations of the coordinates. The weights are
$\omega=0.125939180544827.$
\item The points defined by the barycentric coordinates $(a,b,b)$,
$$a=0.059715871789770, \qquad a+2b=1$$
and cyclic permutations of the coordinates. The weights are
$\omega=0.132394152788506.$
\end{itemize}
The integral on the edges are evaluated by the same Gaussian formula as before.
\end{itemize}
We have used the scheme \ref{RD_Stable} where $\Phi_\sigma$ is obtained by the Galerkin formulation using the above quadrature and $\Psi_\sigma$ is defined by \eqref{entropy:sol1} with $\theta=0.01$. The boundary conditions are implemented with a local Lax Friedrich numerical scheme.
The errors are given in table \ref{table:error:1} and \ref{table:error:2}.

\begin{table}[h]
  \begin{center}
  \begin{tabular}{|cc||cc|cc|cc|}\hline
   $h $  & $n_{dofs}$& $L^1$ & slope & $L^2$ & slope & $L^\infty$ & slope\\
  \hline
    $7.06\; 10^{-2}$ & 525 & $9.16\; 10^{-5} $& $-$& $1.39\; 10^{-4}$ & $-$& $1.14\; 10^{-3}$ &$-$ \\
  $3.53\; 10^{-2}$ &2017 & $2.12\; 10^{-5}$ &$2.11$ & $3.30\; 10^{-5}$&$2.07$ & $3.23\; 10^{-4}$ &$1.81$ \\
  $1.76\; 10^{-2}$ &7905& $5.04\; 10^{-6} $&$2.06$ & $7.96\; 10^{-6}$ &$2.04$ & $8.75\; 10^{-5}$ &$1.87$ \\
  $8.82\; 10^{-3}$ &$31297$& $1.51\; 10^{-6} $&$1.74$ & $2.76\; 10^{-6}$ &$1.53$ & $3.13\; 10^{-5}$ &$1.5$ \\
  \hline
  \end{tabular}
  \end{center}
    \caption{\label{table:error:1} Error for problem \eqref{expo}, linear approximation.}
\end{table}

\begin{table}[h]
  \begin{center}
  \begin{tabular}{|cc||cc|cc|cc|}
  \hline
  $h $  & $n_{dofs}$& $L^1$ & slope & $L^2$ & slope & $L^\infty$ & slope\\
  \hline
    $1.41\; 10^{-1}$        & 525   & $1.35\;10^{-5}$ &$-$       &  $2.46 \;10^{-5} $&$-$ & $2.15\;10^{-4} $&$-$ \\
  $7.06 \;10^{-2}$& 2017 & $1.88\;10^{-6}$ &$2.84$ & $3.62 \;10^{-6}$  & $2.76$& $4.08 \;10^{-5}$&$2.39$ \\
  $3.53 \;10^{-2}$& $7905$ & $2.33 \;10^{-7}$ &$3.01$& $4.77 \;10^{-7}$&$2.92$&$ 6.28 \;10^{-6}$&$2.69$ \\
  $1.76 \;10^{-2}$ & $31297$ & $3.12 \;10^{-8}$ &$2.88$ &  $6.35\; 10^{-8}$ &$2.90$ & $8.88\; 10^{-7}$&$2.81$\\
  \hline
  \end{tabular}
  \end{center}
    \caption{\label{table:error:2} Error for problem \eqref{expo}, quadratic approximation.}\end{table}

As expected the optimal orders are reached. We can also observe the improved accuracy for a given number of freedom, between second and third order.

\subsection{Behavior on discontinuous solutions}
The non linear example is an adaptation of the Burgers equation:
\begin{equation}
\label{burger:1}
\begin{array}{cc}
\dpar{u}{y}+\dpar{\sinh u}{x}=0 & \text{if }x\in [0,1]^2\\
u(x,y)=\frac{1}{2}-x &\text{on the inflow boundary}.
\end{array}
\end{equation}
If the $y$-coordinate corresponds to time, we are back to a more standard formulation.

The exact solution consists in a fan that merges into a shock.  We have performed the simulations using a quadratic approximation. The mesh (i.e. all the degrees of freedom) is displayed in figure \ref{mesh}.
\begin{figure}[h]
\begin{center}
\includegraphics[width=0.45\textwidth]{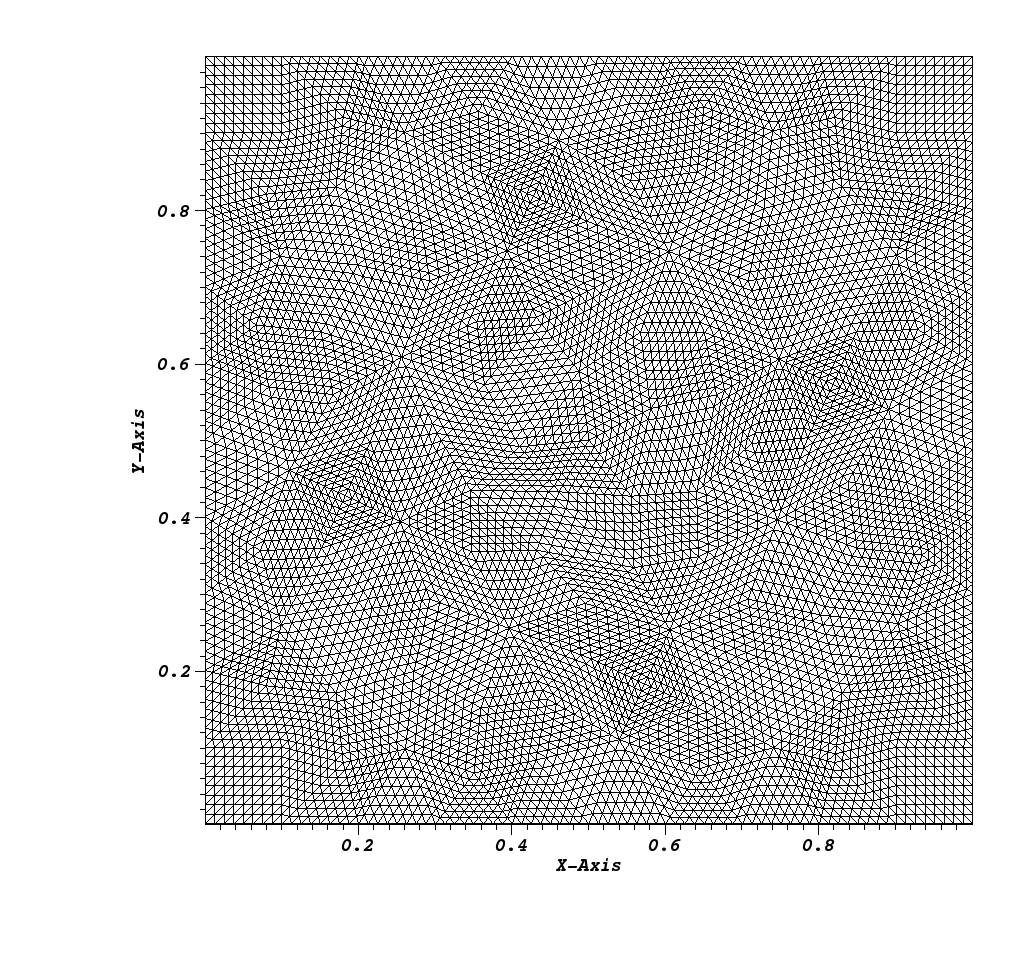}\end{center}
\caption{\label{mesh} Mesh for the numerical experiments. All the Lagrange degrees of freedom are represented. From the  representation, we evaluate the solution at the $\PP^2$ degrees of freedom and then plot the solution using a local $\PP^1$ representation in all the triangles represented on the figure. }
\end{figure}
\begin{figure}[h]
\begin{center}
\subfigure[SUPG-$\theta=0.1$]{\includegraphics[width=0.45\textwidth]{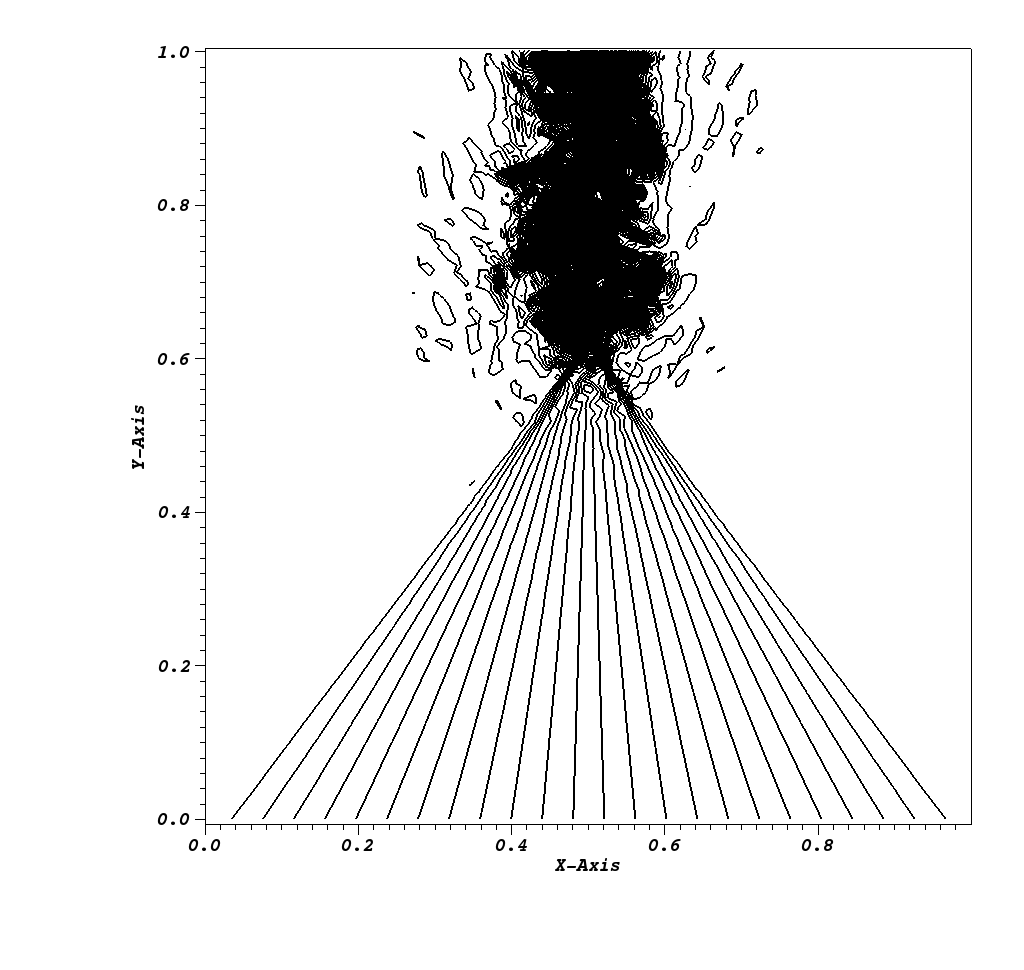}}
\subfigure[SUPG-$\theta=0.01$]{\includegraphics[width=0.45\textwidth]{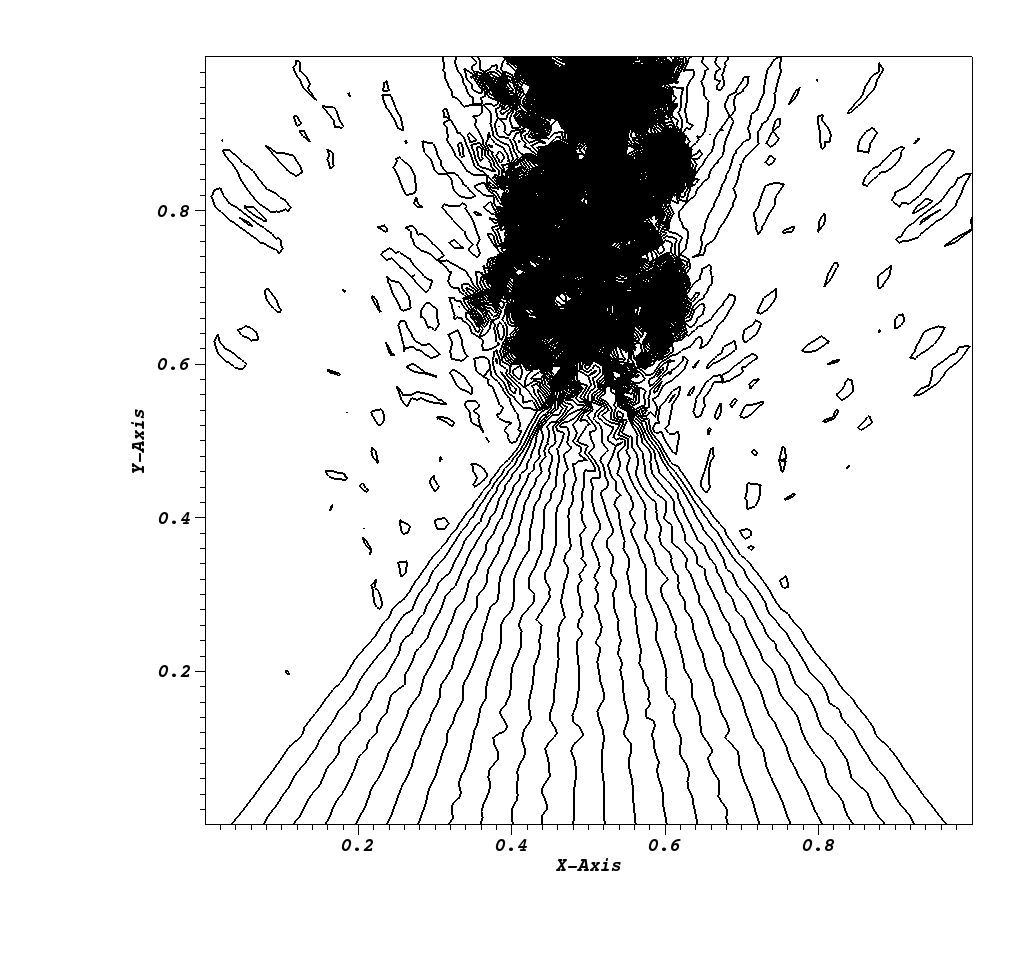}}
\subfigure[SUPG-$\theta=0.02$]{\includegraphics[width=0.45\textwidth]{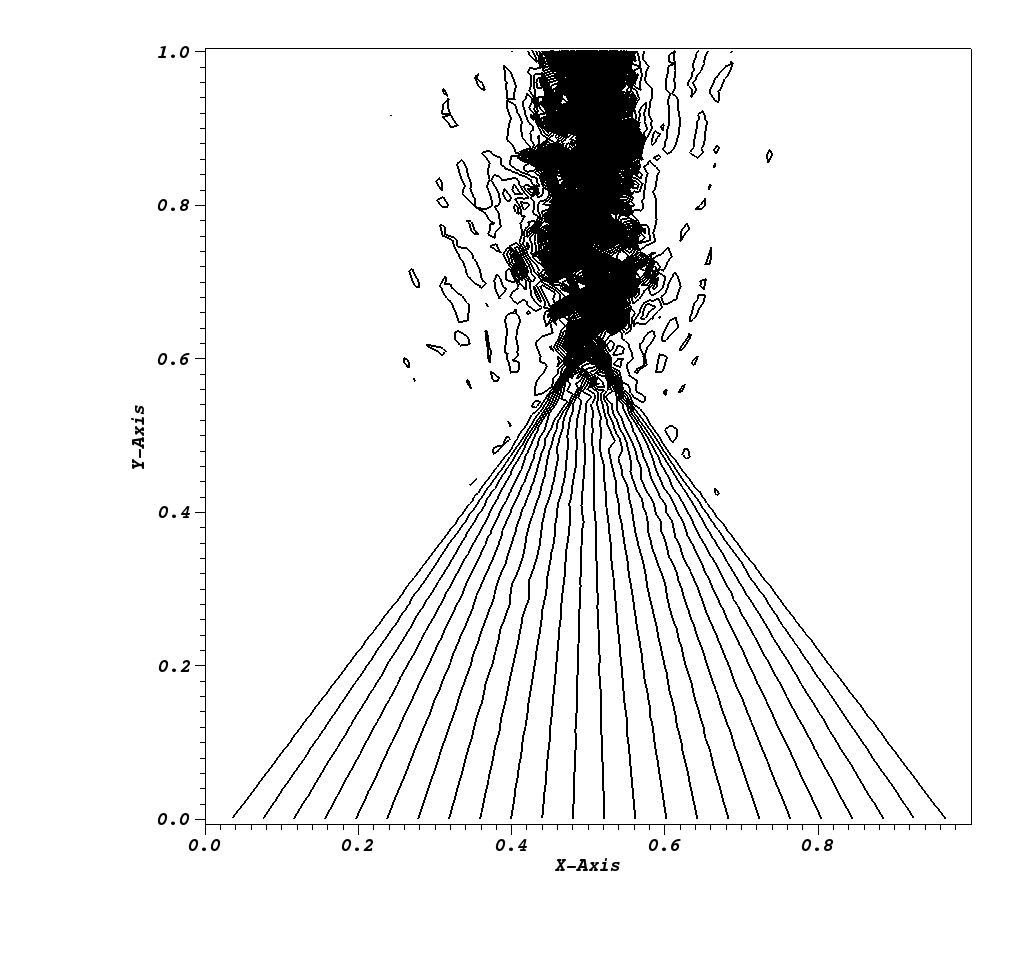}}
\subfigure[GALJ-$\theta=0.01$]{\includegraphics[width=0.45\textwidth]{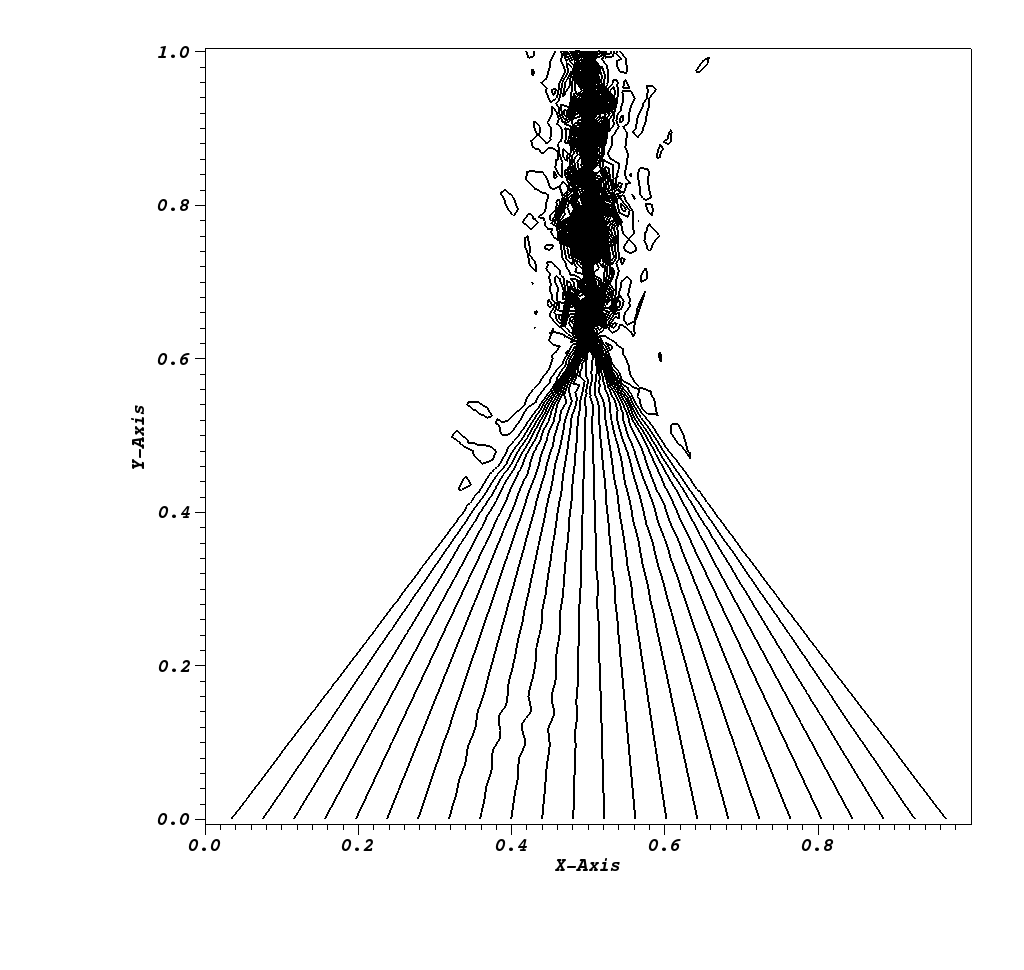}}
\end{center}
\caption{\label{without:Gal} Results without entropy correction, linear schemes}
\end{figure}
\clearpage
 The numerical solutions are obtained by several methods: 
 \begin{itemize}
 \item The SUPG scheme with different values of  $\theta$ in \eqref{entropy:sol2}. 
 \item The Galerkin scheme with jumps and different values of $\theta$ in  \eqref{entropy:sol1}
 \item The non-linear RD scheme where the first order scheme is the Rusanov one (see appendix \ref{RDS})
 \item The non-linear RD scheme with jump filtering, \eqref{schema RDS jump}, the jumps being  tuned by $\theta$
 \item The non-linear RD scheme with streamline  filtering, \eqref{schema RDS SUPG}, the jumps being  tuned by $\theta$
 \end{itemize}
 These schemes are nicknamed respectively  as SUPG-$\theta$, GalJ-$\theta$,  RD, RD-J-$\theta$ and RD-S-$\theta$. We also have made a second set of simulations where the schemes are modified by adding the entropy correction. More precisely, this is done after the evaluation of the Galerkin term for the SUPG-$\theta$ and GALJ-$\theta$ schemes, and after the nonlinear RD step of the RD-J-$\theta$ and RD-S-$\theta$ steps, i.e. before the filtering step in all cases, as described in section \ref{sec:construction}. These schemes are nicknamed as  SUPG-E-$\theta$, GalJ-E-$\theta$,  RD, RD-J-E-$\theta$ and RD-S-E-$\theta$ respectively. Since the numerical flux is not polynomial, the Galerkin step is only approximate: the entropy correction is a priori acting. We have used the same quadrature points as in section \ref{sec:accuracy}.
 
  Our goal is to see if we can lower the parameter $\theta$ in the filtering term because we have a fine tuning of the local entropy production. All the simulations are done with the Euler forward time stepping, CFL=$0.3$. The plots use the same isolines between $2$ and $-2$, 100 isolines: we can see the numerical oscillations, if any.
 
\begin{figure}[h]
\begin{center}
\subfigure[RD scheme]{\includegraphics[width=0.45\textwidth]{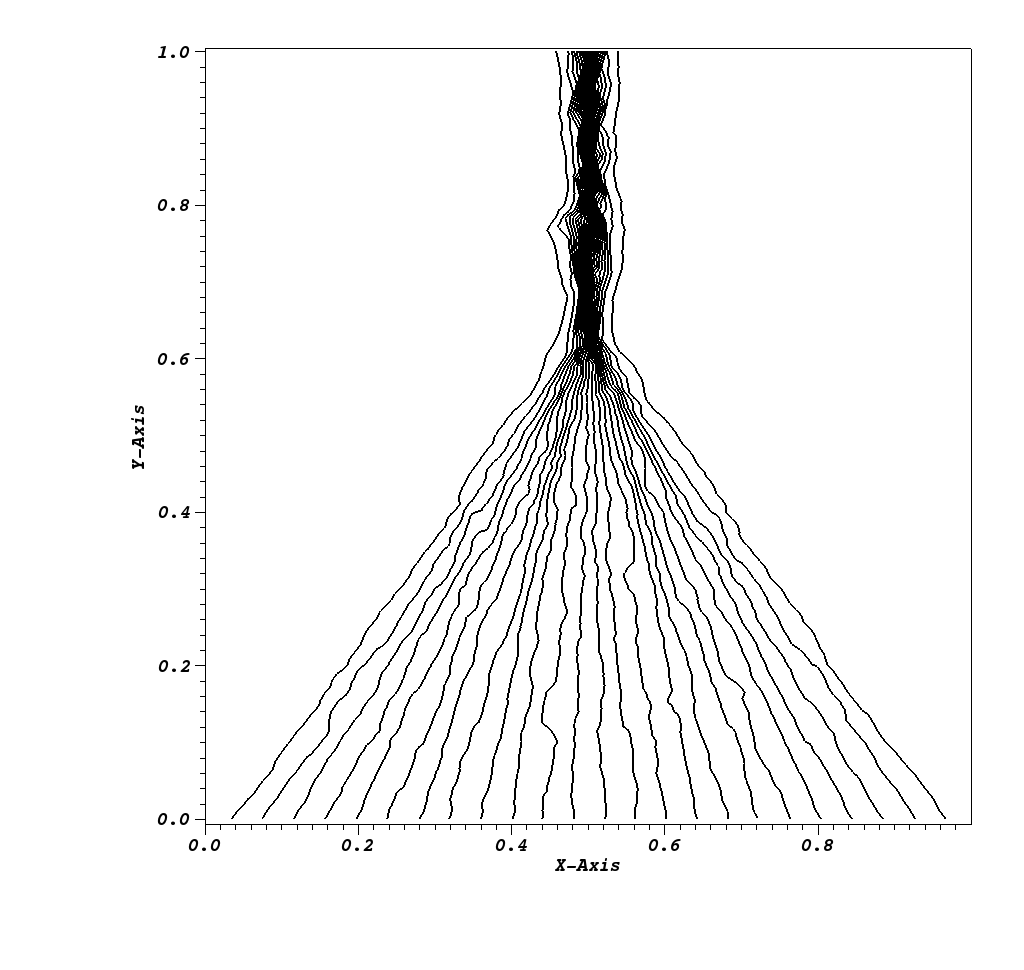}}
\subfigure[RDJ-$0.1$]{\includegraphics[width=0.45\textwidth]{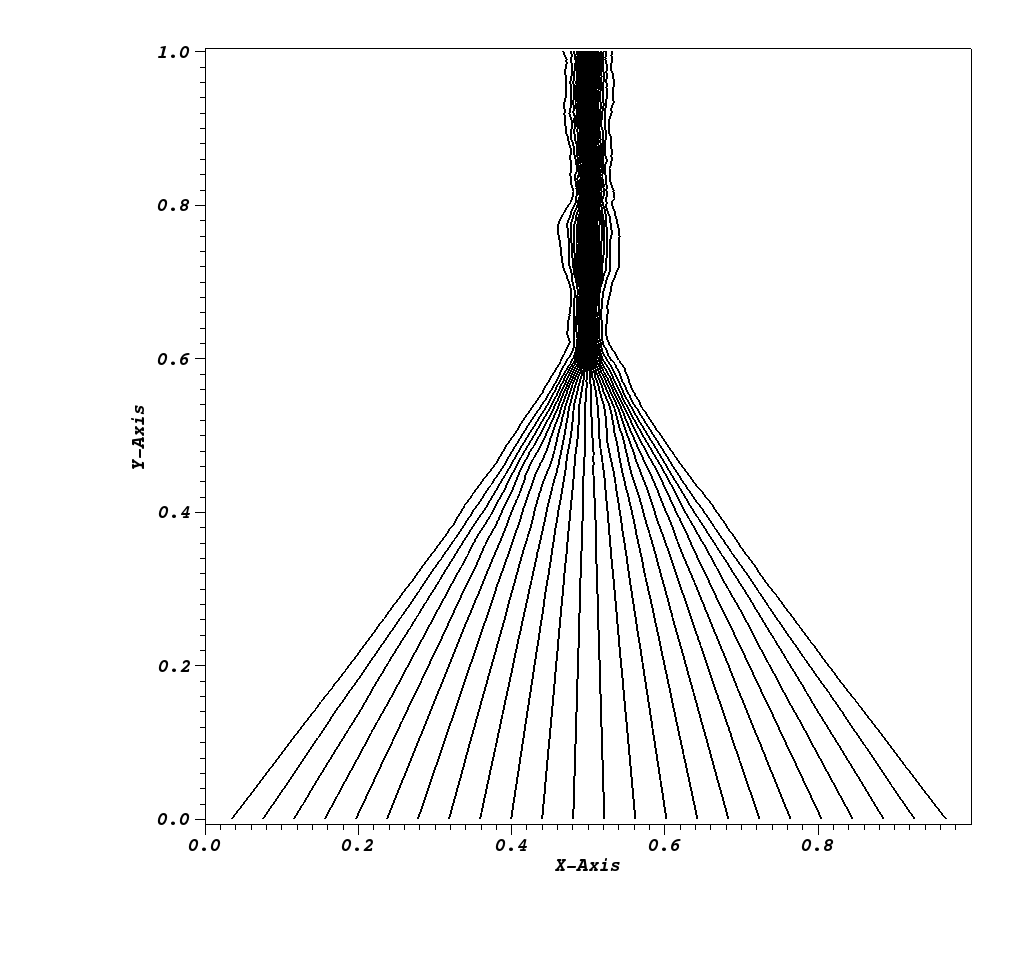}}
\subfigure[RDJ-$0.02$]{\includegraphics[width=0.45\textwidth]{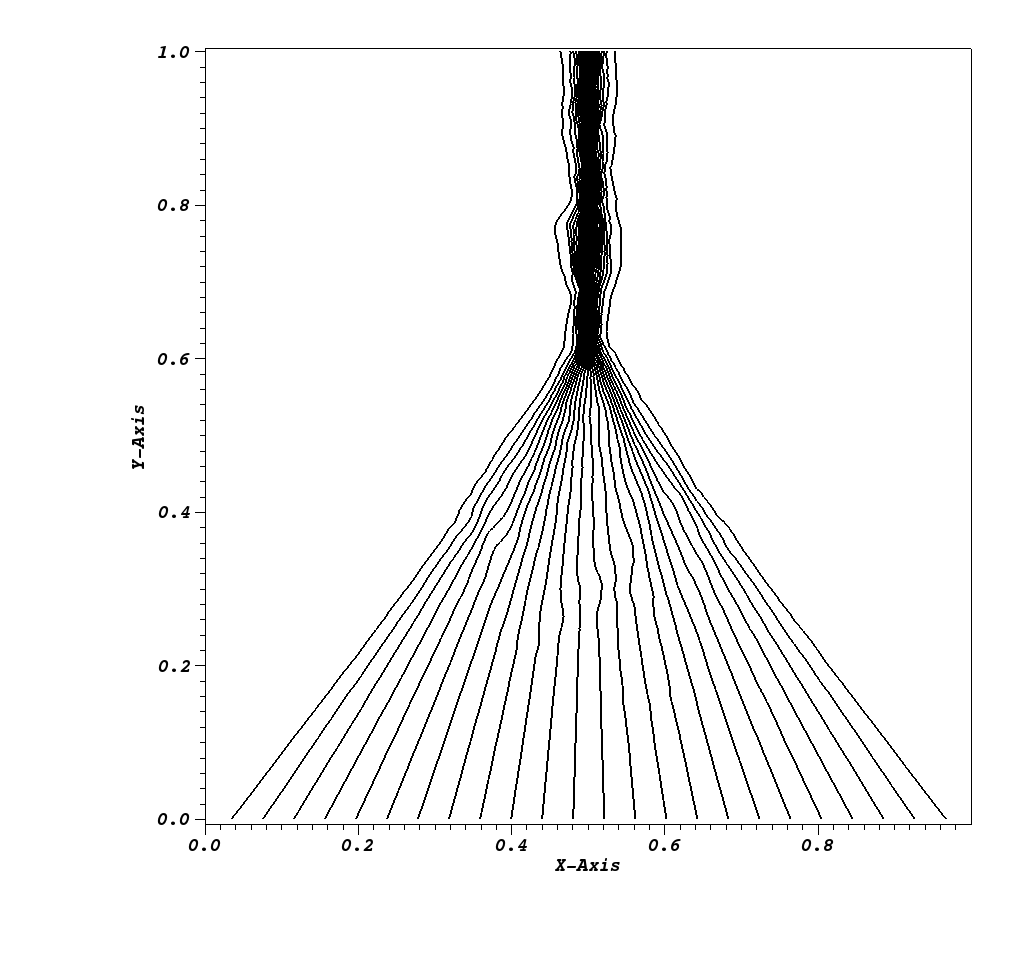}}
\subfigure[RDJS-$0.02$-$0.02$]{\includegraphics[width=0.45\textwidth]{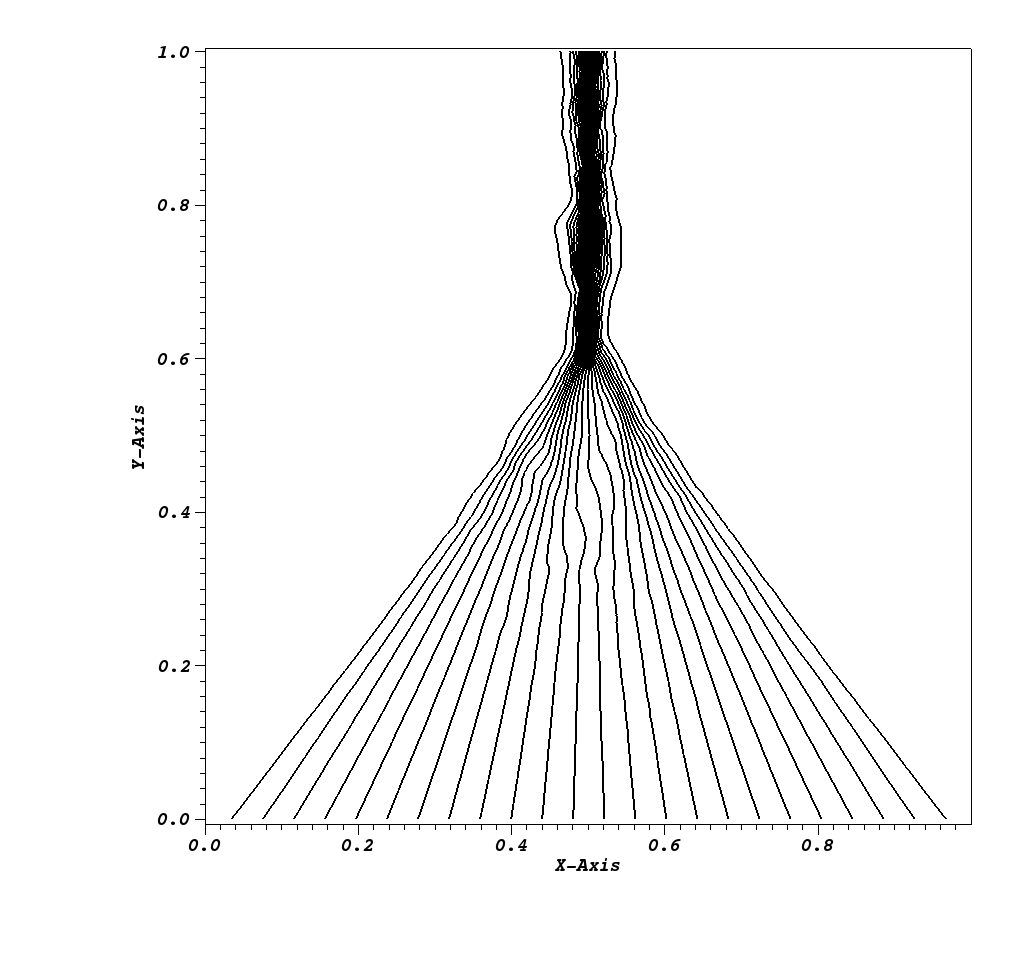}}
\end{center}
\caption{\label{without:RD} Results without entropy correction, non linear schemes}
\end{figure}
\clearpage
From figure \ref{without:Gal}, we see that the SUPG scheme without correction and $\theta_S=0.1$ provides very good results in the fan but is almost unstable in the shock (the calculation has been stopped since it is meaningless). When we lower the parameter $\theta$ to $0.01$, the shock does not improve, and some wiggles exist in the fan, so we are not filtering enough. A relatively good compromise is reached for $\theta=0.02$. This is in contrast with the GalJ-0.01 result that are comparable to the SUPG-0.02 ones (with smaller oscillations at the shock, by the way).

From figure \ref{without:RD}, we see that $\theta=0$ leads to wiggles in the fan. They are due to spurious modes, and not a stability problem, see \cite{remi_ENORD,waterloo}. We have tried several values of $\theta$ for the streamline and the jump term, it seems that the best results are obtained for $\theta=0.1$, both for the fan and the discontinuity (remember that we have $100$ isolines between $-2$ and $2$, so $25$ between $-0.5$ and $0.5$.

From figure \ref{with}, we see that Gal-E-$0.01$ behaves much better that SUPG-E-$0.01$. This is an other indication (comparing to figure \ref{without:Gal}-(b) that the streamline term is not suffisant. The solution is much better with $\theta=0.02$, see figure \ref{with}-(c), but it is similar to \ref{without:Gal}-(c). We have also experienced similar disapointments with the unsteady version of these schemes in \cite{Abgrall2017} where the jump term seems to filter out more efficiently.
 Comparing figures \ref{without:RD}-(c) and \ref{with}-(d), we see that the quality of the fan has improved  with the entropy correction, while, of course, the shock becomes wiggly. If we decrease the value of $\theta$, then the solution becomes similar to \ref{with}-(b).
\begin{figure}[h]
\begin{center}
\subfigure[GALJ-E-$\theta=0.01$]{\includegraphics[width=0.45\textwidth]{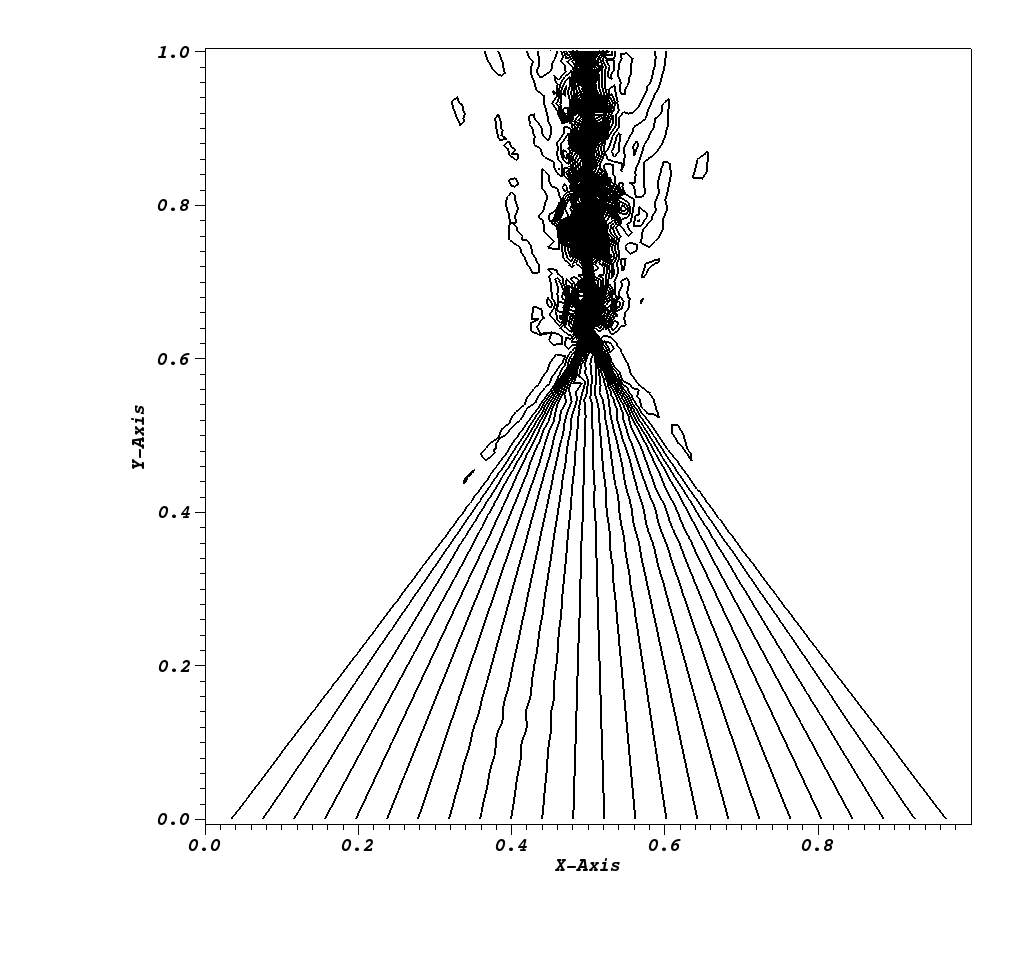}}
\subfigure[SUPG-E-$\theta=0.01$]{\includegraphics[width=0.45\textwidth]{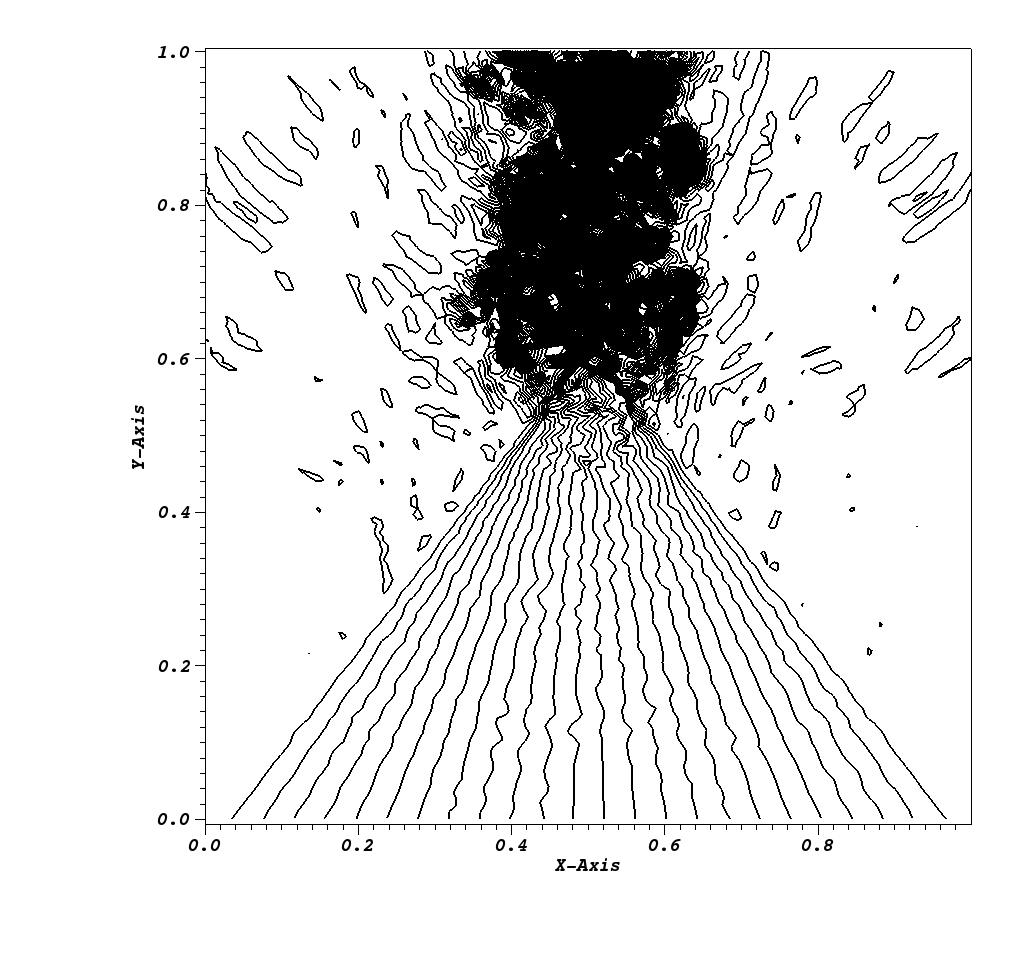}}
\subfigure[SUPG-E-$\theta=0.02$]{\includegraphics[width=0.45\textwidth]{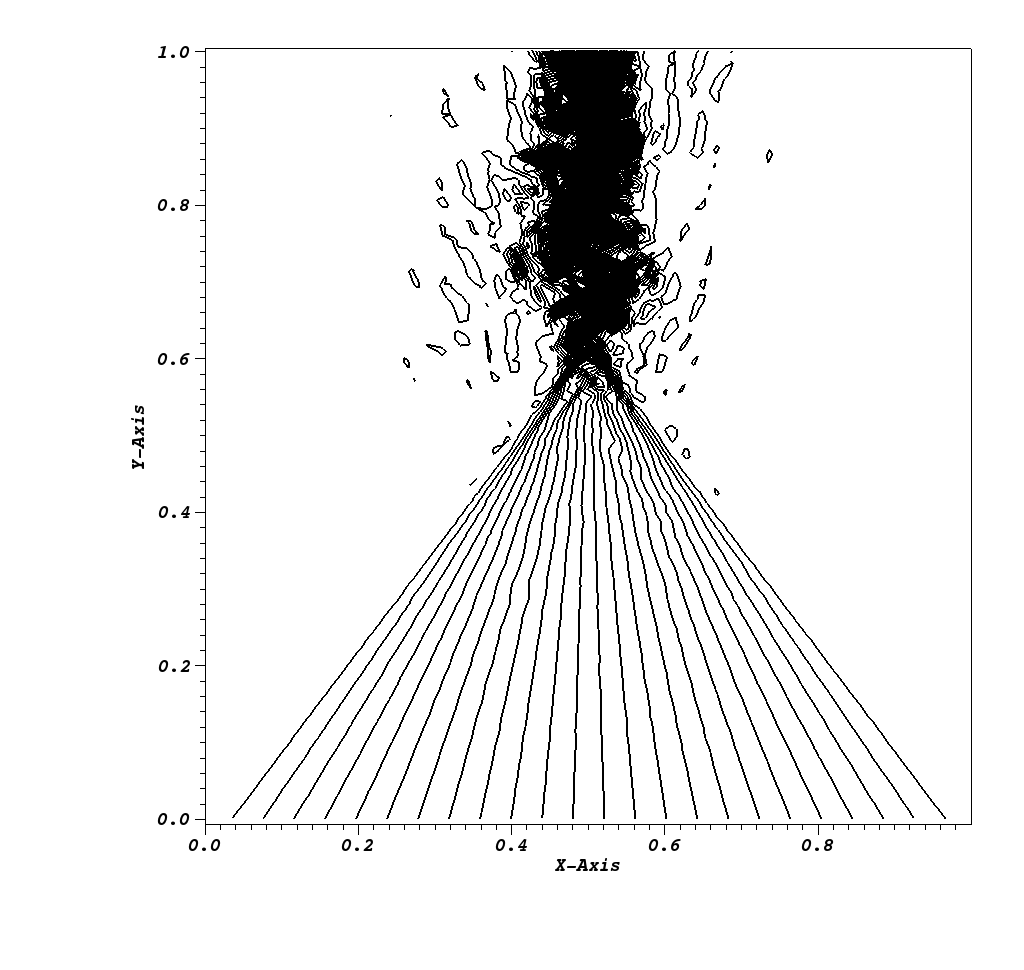}}
\subfigure[RDJ-E-$0.02$]{\includegraphics[width=0.45\textwidth]{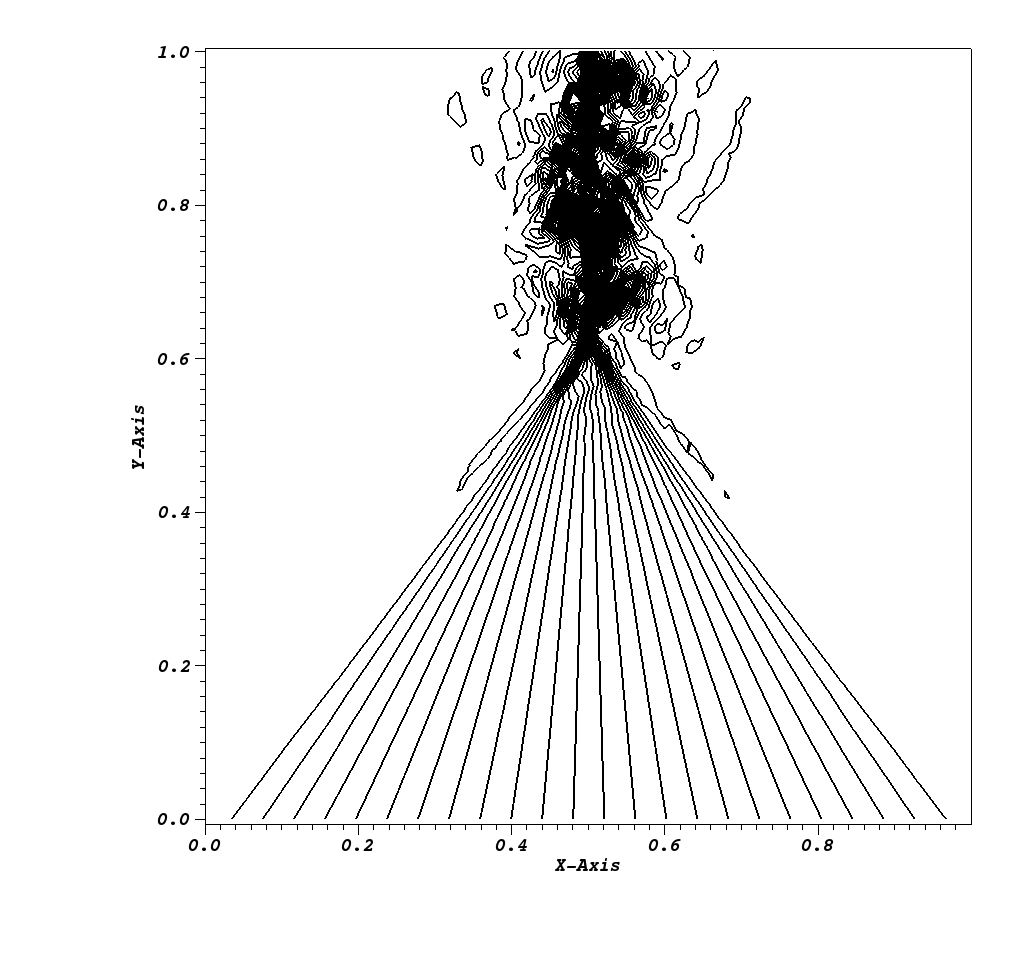}}
\end{center}
\caption{\label{with}Results with  with the entropy term}
\end{figure}
\clearpage

}

\section{Conclusion}
We have developed a general technique to guaranty that the discretisation of a conservative hyperbolic problem is consistent with an additional conservation relation. This is a generalisation of the work \cite{paola} and a sequel of \cite{abgrall:hal-01573592}.  This construction  uses  in depth a reinterpretation of conservative schemes as Residual distribution schemes. 

\red{Of course, the modifications will  depend on the problem. Aiming at illustrating this idea, this is why  we have focussed on an entropy equality and have constructed explicit modifications of the original discretisation that guaranty a consistency with the original PDE and the additional entropy inequality.  We also have shown that independently of the choice of the entropy flux, one lose one order of accuracy. However, in the case of discontinuous Galerkin schemes, and stabilized residual distribution schemes, we can provide an explicit form of the entropy flux that does not lead to a degradation of the accuracy. To achieve this goal, we do not need any particular structure of the quadrature formula, in particular they do not need to be exact nor to achieve the summation-by-part property. }

\red{We have tested the method on scalar problem, the formula we have given also work for systems. We have shown, numerically, at least for finite element like methods, that the optimal accuracy is kept. We have also checked that the method is globally conservative for the original conserved variable and the entropy. Using the results of \cite{abgrall:hal-01573592}, it is easy to see that, at least for the steady case, that these schemes can be reformulated as finite volume schemes. In the case of entropy stable scheme, we can get an \emph{explicit} form of the numerical flux for the conserved variable and the entropy.}

\red{We have also shown that adding this entropy correction can help to lower the amount of artificial diffusion that is needed to filter spurious modes for the Galerkin method and the non linear residual distribution ones. Our study also confirm that the stabilisation by jump is more efficient than the streamline diffusion one. One reason is probaly that the jump stabilisation allow inter-element exchange of informations, contrarily to the streamline one.}


\medskip

\red{The emphasis of this paper is put on the steady case, but the unsteady state is similar, see \cite{Mario} and \cite{Abgrall2017,HighFluid2018}.  
However, if  the unsteady case can be considered as well, this is to the price of having implicit formulations, in the spirit of \cite{paola}. This issue will be considered in future work. Future work will also deal with the question of how to have an entropy conservative scheme in smooth regions, and how to guaranty non oscillatory behaviour in discontinuous ones }

\bigskip

We conclude this work in noticing that one of the ingredients of the method is that any element has at least 2 vertices: the conservation relation translates by a linear relations on the residuals, and the modification by another one. Having more than 2 degrees of freedom, one solution is a priori possible. As a consequence, it could be  possible, using the same kind of technique, to guaranty more that one additional conservation relation, this will also be considered for future work. 
\section*{Acknowledgements}
The author has been funded in part  by the SNSF project 200021\_153604 "High fidelity simulation for compressible materials". I would also like to thanks Anne Burbeau (CEA-DEN) for her critical reading of the first draft of this paper. Her input has hopefully helped to improve the readability of this paper. 
I also warmly thanks the three anonymous referees whose comments and criticisms have, I hope, improved the original version of the paper. The remaining mistakes are mine.

\medskip

A personal note to end: though the author is curently editor in chief of this Journal, the EES system has been parametrised in such a way that I, of course,  have absolutely no access to the identity of the referees, nor to the identity of  the editor in charge of the paper. This is, of course, the only paper of this type. Hence I would like to thanks the Elsevier team for making this singular situation possible. 
\bibliographystyle{unsrt}
\bibliography{papier}

\begin{thebibliography}{10}

\bibitem{ShuEntropy}
T.~Chen and C.-W. Shu.
\newblock {Entropy stable high order discontinuous {Galerkin} methods with
  suitable quadrature rules for hyperbolic conservation laws}.
\newblock {\em J. Comput. Phys.}, 345:427 -- 461, 2017.

\bibitem{Mishra}
A.~Hiltebrand and S.~Mishra.
\newblock {Entropy stable shock capturing space–time discontinuous {Galerkin}
  schemes for systems of conservation laws}.
\newblock {\em Numer. Math.}, 126:103--151, 2014.

\bibitem{Gassner}
G.J. Gassner.
\newblock A skew-symmetric discontinuous {Galerkin} spectral element
  discretisation and its relation to {SBP-SAT} finite differennce methods.
\newblock {\em SIAM J. Sci. Comput.}, 35:A1233--A1253, 2013.

\bibitem{Zingg}
J.E. Hicke, D.C.~Del Rey, and D.W. Zingg.
\newblock Multidimensional summation-by-part operators: general theory and
  application to simplex elements.
\newblock {\em SIAM J. Sci. Comput.}, 38:A1935--A1958, 2016.

\bibitem{TadmorB}
E.~Tadmor.
\newblock {\em Handbook of Numerical Methods for Hyperbolic Problems, Basic and
  fundemental issues}, volume~17 of {\em Handbook of Numerical Analysis},
  chapter Entropy Stable Schemes, pages 467--490.
\newblock North-Holland, 2016.

\bibitem{TadmorVieux}
E.~Tadmor.
\newblock The numerical viscosity of entropy stable schemes for systems of
  conservation laws, {I}.
\newblock {\em Math. Comp.}, 49:91--103, 1987.

\bibitem{TadmorActa}
E.Tadmor.
\newblock Entropy stability theory for difference approximations of nonlinear
  conservation laws and related time-dependent problems.
\newblock {\em Acta Numerica}, 13:451--512, 2003.

\bibitem{JiangShu}
G.S. Jiang and C.-W. Shu.
\newblock On a cell entropy inequality for discontinuous {Galerkin} methods.
\newblock {\em Math. Comput.}, 62:531--53, 1994.

\bibitem{bouchut}
F.~Bouchut, C.~Bourdarias, and B.~Perthame.
\newblock A {MUSCL} method satisfying all the numerical entropies.
\newblock {\em Math. Comp.}, 65(1439-1461), 1996.

\bibitem{LeFlochMercier}
P.G. Lefloch, J.-M. Mercier, and C.~Rohde.
\newblock Fully discrete, entropy conservative schemes of arbitrary order.
\newblock {\em SIAM J. Numer. Anal.}, 40:1968--1992, 2002.

\bibitem{HouLiu}
S.~Hou and X.-D. Liu.
\newblock Solutions of multi-dimensional hyperbolic systems of conservation
  laws by square entropy condition satisfying discontinuous {Galerkin} method.
\newblock {\em J. Sci. Comput.}, 31:127--151, 2007.

\bibitem{Ismail}
F.~Ismail and P.L. Roe.
\newblock Affordable, entropy-consistent {Euler} flux functions {II}: entropy
  production at shocks.
\newblock {\em J. Comput. Phys.}, 228:5410--5436, 2009.

\bibitem{Guermond}
J.-L. Guermond, R.~Pasquetti, and B.~Popov.
\newblock Entropy viscosity method for nonlinear conservation laws.
\newblock {\em J. Comput. Phys.}, 230:4248--4267, 2011.

\bibitem{Oefner}
H.~{Ranocha}, Ph. {\"Offner}, and Th. {Sonar}.
\newblock {Summation-by-parts operators for correction procedure via
  reconstruction.}
\newblock {\em {J. Comput. Phys.}}, 311:299--328, 2016.

\bibitem{Carpenter1}
M.~{Parsani}, M.~H. {Carpenter}, T.~C. {Fisher}, and Eric~J. {Nielsen}.
\newblock {Entropy stable staggered grid discontinuous spectral collocation
  methods of any order for the compressible Navier-Stokes equations.}
\newblock {\em {SIAM J. Sci. Comput.}}, 38(5):a3129--a3162, 2016.

\bibitem{Nordstrom1}
M.~{Sv\"ard} and J.~{Nordstr\"om}.
\newblock {Review of summation-by-parts schemes for initial-boundary-value
  problems.}
\newblock {\em {J. Comput. Phys.}}, 268:17--38, 2014.

\bibitem{hiltebrand14:_entrop_galer}
A.~Hiltebrand and S~Mishra.
\newblock Entropy stable shock capturing space–time discontinuous {Galerkin}
  schemes for systems of conservation laws.
\newblock {\em Numer. Math.}, 126(1):103--151, 2014.

\bibitem{svard}
M.~Sv\"ard and H.~\"Ozcan.
\newblock Entropy-stable schemes for the {Euler} equations with far-field and
  wall boundary conditions.
\newblock {\em J. Sci. Comput.}, 58:61--89, 2014.

\bibitem{zbMATH06599705}
Deep {Ray} and Praveen {Chandrashekar}.
\newblock {Entropy stable schemes for compressible {Euler} equations.}
\newblock {\em {Int. J. Numer. Anal. Model., Ser. B}}, 4(4):335--352, 2013.

\bibitem{zbMATH06799652}
Deep {Ray}, Praveen {Chandrashekar}, Ulrik~S. {Fjordholm}, and Siddhartha
  {Mishra}.
\newblock {Entropy stable scheme on two-dimensional unstructured grids for
  {Euler} equations.}
\newblock {\em {Commun. Comput. Phys.}}, 19(5):1111--1140, 2016.

\bibitem{GASSNER2016291}
Gregor~J. Gassner, Andrew~R. Winters, and David~A. Kopriva.
\newblock A well balanced and entropy conservative discontinuous {Galerkin}
  spectral element method for the shallow water equations.
\newblock {\em Applied Mathematics and Computation}, 272:291 -- 308, 2016.
\newblock Recent Advances in Numerical Methods for Hyperbolic Partial
  Differential Equations.

\bibitem{chandrasekar}
P.~Chandrasekar.
\newblock Kinetic energy preserving and entropy stable finite volume scheme for
  compressible {{Euler}} and {Navier-Stokes} equations.
\newblock {\em Commun. Comput. Phys.}, 14:1252--1286, 2013.

\bibitem{ciarlet}
P.~Ciarlet.
\newblock {\em The finite element method for elliptic problems}.
\newblock North-Holland, Amsterdam, 1978.

\bibitem{ErnGuermond}
A.~Ern and J.L. Guermond.
\newblock {\em Theory and practice of finite elements}, volume 159 of {\em
  Applied Mathematical Sciences}.
\newblock Springer verlag, 2004.

\bibitem{abgrallLarat}
R.~Abgrall, A.~Larat, and M.~Ricchiuto.
\newblock {Construction of very high order residual distribution schemes for
  steady inviscid flow problems on hybrid unstructured meshes.}
\newblock {\em J. Comput. Phys.}, 230(11):4103--4136, 2011.

\bibitem{abgralldeSantisSISC}
R.~Abgrall and D.~de~Santis.
\newblock High-order preserving residual distribution schemes for
  advection-diffusion scalar problems on arbitrary grids.
\newblock {\em SIAM J. Sci. Comput.}, 36(3):A955--A983, 2014.
\newblock also \url{http://hal.inria.fr/docs/00/76/11/59/PDF/8157.pdf}.

\bibitem{Hughes1}
T.J.R. Hughes, L.P. Franca, and M.~Mallet.
\newblock A new finite element formulation for {CFD}: {I}. symmetric forms of
  the compressible {E}uler and {N}avier-{S}tokes equations and the second law
  of thermodynamics.
\newblock {\em Comput. Methods. Appl. Mech. Engrg.}, 54:223--234, 1986.

\bibitem{burman}
E.~Burman and P.~Hansbo.
\newblock Edge stabilization for {G}alerkin approximation of
  convection-diffusion-reaction problems.
\newblock {\em Comput. Methods Appl. Mech. Engrg}, 193:1437--1453, 2004.

\bibitem{energie}
R.~{Abgrall}.
\newblock {Essentially non-oscillatory residual distribution schemes for
  hyperbolic problems.}
\newblock {\em {J. Comput. Phys.}}, 214(2):773--808, 2006.

\bibitem{AbgrallRoe}
R.~Abgrall and P.~L. Roe.
\newblock {High-order fluctuation schemes on triangular meshes.}
\newblock {\em J. Sci. Comput.}, 19(1-3):3--36, 2003.

\bibitem{kroner}
D.~Kr\"oner, M.~Rokyta, and M.~Wierse.
\newblock A {Lax-Wendroff} type theorem for upwind finite volume schemes in
  $2$-d.
\newblock {\em East-West J. Numer. Math.}, 4(4):279--292, 1996.

\bibitem{abgrall:hal-01573592}
R.~Abgrall.
\newblock Some remarks about conservation for residual distribution schemes.
\newblock {\em Computational Methods in Applied Mathematics}, 2017.
\newblock published online 2017-12-06,
  doi:\url{https://doi.org/10.1515/cmam-2017-0056}.

\bibitem{svetlana}
R.~Abgrall and S.~Tokareva.
\newblock Staggered grid residual distribution scheme for lagrangian
  hydrodynamics.
\newblock {\em SIAM J. Sci. Comput}, 39(5):{A2317--A2344}, 2017.
\newblock see also \url{ https://hal.inria.fr/hal-01327473}.

\bibitem{paola}
R.~Abgrall, P.~Baccigalupi, and S.~Tokareva.
\newblock A high-order nonconservative approach for hyperbolic equations in
  fluid dynamics.
\newblock {\em Computers and Fluids},
  doi:\url{https://doi.org/10.1016/j.compfluid.2017.08.019}, 2018.
\newblock see also \url{https://hal.archives-ouvertes.fr/hal-01476636v1}.

\bibitem{remi_ENORD}
R.~Abgrall.
\newblock Essentially non oscillatory residual distribution schemes for
  hyperbolic problems.
\newblock {\em J. Comput. Phys.}, 214(2):773--808, 2006.

\bibitem{waterloo}
R.~Abgrall.
\newblock About non linear stabilization for scalar hyperbolic problems.
\newblock \url{https://hal.archives-ouvertes.fr/hal-01572473}, July 2016.

\bibitem{Abgrall2017}
R.~Abgrall.
\newblock High order schemes for hyperbolic problems using globally continuous
  approximation and avoiding mass matrices.
\newblock {\em J. Sci. Comput.}, 73, 2017.
\newblock \url{https://hal.archives-ouvertes.fr/hal-01445543v2}.

\bibitem{Mario}
M.~Ricchiuto and R.~Abgrall.
\newblock {Explicit Runge-Kutta residual distribution schemes for time
  dependent problems: second order case.}
\newblock {\em J. Comput. Phys.}, 229(16):5653--5691, 2010.

\bibitem{HighFluid2018}
R.~Abgrall, P.~Bacigaluppi, and S.~Tokareva.
\newblock High-order residual distribution scheme for the time-dependent
  {Euler} equations of fluid dynamics.
\newblock {\em Computer \& Mathematics with Applications}, 2018.
\newblock doi: \url{https://doi.org/10.1016/j.camwa.2018.05.009}.

\bibitem{icm}
R.~Abgrall.
\newblock On a class of high order schemes for hyperbolic problems.
\newblock In {\em Proceedings of the International Conference of
  Mathematicians, volume {IV}}, pages 699--726, Seoul, 2014.

\end{thebibliography}
\appendix
\section{Construction of the non linear stabilisation for RD schemes}\label{RDS}
Here we consider a globally continuous approximation: $\bu^h \in \mathcal{V}_h\cap C^0(\Omega)$.

Consider one element $K$. Since there is no ambiguity, the drop, for the residuals,  any reference to $K$ in the following. The total residual is defined by
$$
\Phi(\bu^h)=\int_{\partial K} \bF(\bu^h)\cdot \bn \; d\gamma,$$
and we assume to have monotone residuals $\{\Phi_\sigma^L\}_{\sigma\in K}$. By this we mean 
$$\Phi_\sigma^L(\bu^h)=\sum\limits_{\sigma'\in K} c_{\sigma\sigma'}^L(\bu_\sigma-\bu_{\sigma'})$$
with $c_{\sigma\sigma'}\geq 0$ that also satisfies
$$\sum\limits_{\sigma\in K}\Phi_\sigma^L=\Phi.$$ It can easily be shown  that the condition $c_{\sigma\sigma'}\geq 0$ garanties that the scheme is monotone under a CFL like condition. One example is 
given by the Rusanov residuals:
$$\Phi_\sigma^{Rus}(\bu^h)=-\int_K \nabla \varphi_\sigma \cdot  \bF(\bu^h)\; d\bx+\int_{\partial K} \varphi_\sigma\bbf(\bu^h)\cdot \bn\; d\gamma+\alpha (\bu_\sigma-\bar \bu),
$$
where $\bar \bu$ is the arithmetic average of of the $\bu_\sigma's$ on $K$ and $\alpha$ satisfies:
$$\alpha\geq \#K\; \max_{\sigma, \sigma'\in K} \bigg | \int_K \varphi_\sigma \nabla_\bu\bF(\bu^h)\;  d\bx \bigg |.$$
Here $\#K$ is the number of degrees of freedom in $K$.
Indeed, this residual can be rewritten as 
$$\Phi_\sigma^{Rus}(\bu^h)=\sum_{\sigma'\in K} c_{\sigma\sigma'}(\bu_\sigma-\bu_{\sigma'})$$
with
$$c_{\sigma\sigma'}=\int_K \varphi_\sigma \cdot \nabla_\bu\bF(\bu^h)\; d\bx+\dfrac{\alpha}{\#K}.$$
Under the condition above, $c_{\sigma\sigma'}\geq 0$ and hence we have a maximum principle.

The coefficients $\beta_\sigma$ introduced in the relations \eqref{schema RDS SUPG} and \eqref{schema RDS jump} are defined by:
$$
\beta_\sigma=\dfrac{\max(0,\frac{\Phi_\sigma^{L}}{\Phi})}
{
\sum\limits_{\sigma'\in K} \max(0,\frac{\Phi_{\sigma'}^{L}}{\Phi})}.
$$ and can be shown to be always defined, to  guaranty a local maximum principle for \eqref{schema RDS SUPG} and \eqref{schema RDS jump}, see \cite{abgrallLarat}.

\begin{remark}[About the coefficients $c_{\sigma\sigma'}^L$]
All the examples of monotone residual we are aware of are such that for linear problems, te $c_{\sigma\sigma'}^L$ are independant of $\bu^h$. Then one can show that for any $\sigma\in K$,
$$\sum\limits_{\sigma'\in K}\big ( c_{\sigma\sigma'}^L-c_{\sigma'\sigma}^L\big )=\int_K \nabla\bbf(\bu^h)\cdot\nabla \varphi_\sigma d\bx.$$
This relation implies the conservation relation \eqref{conservation:K}.
\end{remark}
\section{Examples of finite volume writing of Residual distribution schemes}
\label{control}
{
This part is taken from \cite{abgrall:hal-01573592}.  As a motivation, we first show that any finite volume scheme can be written as  a residual distribution sccheme. Then we give two example of the converse. The general statement can be found in \cite{abgrall:hal-01573592}, it applies also to discontinuous representations, see the same reference

\subsection{Finite volume as Residual distribution schemes}
 The notations are defined in Figure \ref{fig:fv}.
\begin{figure}[h]
\begin{center}
\subfigure[]{\includegraphics[width=0.45\textwidth]{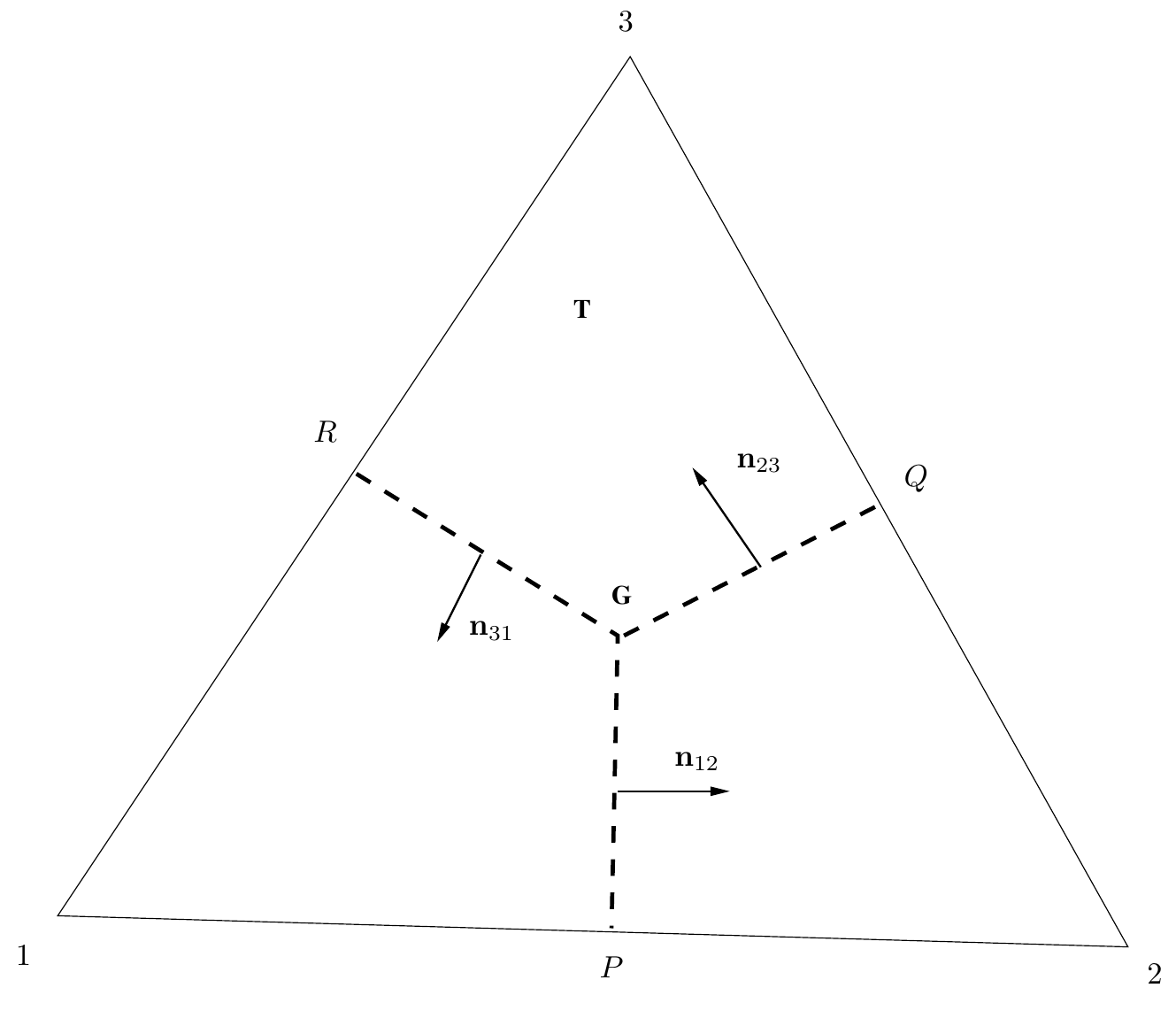}}
\subfigure[]{\includegraphics[width=0.45\textwidth]{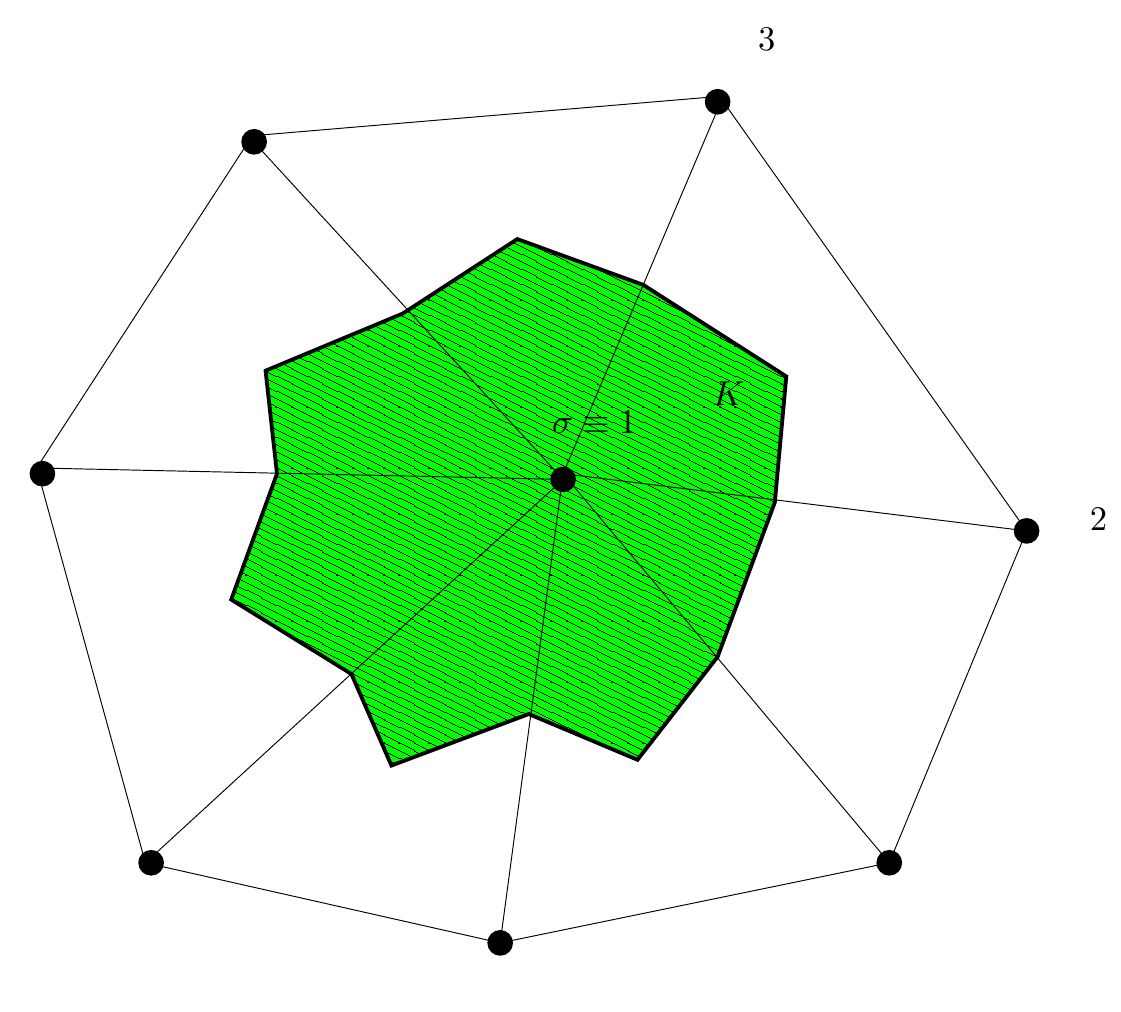}}
\end{center}
\caption{\label{fig:fv} Notations for the finite volume schemes. On the left: definition of the control volume for the degree of freedom $\sigma$.
 The vertex $\sigma$ plays the role of the vertex $1$ on the left picture for the triangle K. The control volume $C_\sigma$ associated to $\sigma=1$ is green on the right and corresponds to $1PGR$ on the left. The vectors $\bn_{ij}$ are normal to the internal edges scaled by the corresponding edge length}
\end{figure}
We specialize ourselves to the case of triangular elements, but  \emph{exactly the same arguments} can be given for more general elements,
 provided a conformal approximation space can be constructed. This is  the case for
triangle elements, and we can take $k=1$.

The control volumes in this case are defined as the median cell, see figure \ref{fig:fv}.
We concentrate on the  approximation of $\text{div } \bbf$, see equation \eqref{eq1}.   Since the boundary of $C_\sigma$ is a closed polygon, the scaled outward normals $\bn_\gamma$ to $\partial C_\sigma$ sum up to 0:
$$
\sum_{\gamma \subset \partial C_\sigma}\bn_\gamma=0$$
where $\gamma$ is any of the segment included in $\partial C_\sigma$, such as $PG$ on Figure \ref{fig:fv}. 
Hence
\begin{equation*}
\begin{split}
\sum_{\gamma \subset \partial C_\sigma} \hbbf_{\bn_\gamma }(\bu_\sigma , \bu_-& )= \sum_{\gamma \subset \partial C_\sigma} \hbbf_{\bn_\gamma }(\bu_\sigma, \bu_- )- \bigg (\sum_{\gamma \subset \partial C_\sigma}\bn_\gamma\bigg )\cdot \bbf (\bu_\sigma)\\
&=\sum\limits_{K, \sigma\in K} \sum\limits_{\gamma \subset \partial C_\sigma\cap K} \big ( \hbbf_{\bn_\gamma }(\bu_\sigma, \bu_- )-\bbf (\bu_\sigma)\cdot \bn_\gamma \big )
\end{split}
\end{equation*}
To make things explicit, in $K$, the internal boundaries are $PG$, $QG$ and $RG$, and those around $\sigma\equiv 1$ are $PG$ and $RG$.
We set
\begin{equation}
\begin{split}
\Phi_\sigma^K(\bu^h)&=\sum\limits_{\gamma\subset \partial C_\sigma\cap K} \big ( \hbbf_{\bn_\gamma }(\bu_\sigma, \bu_- )-\bbf (\bu_\sigma)\cdot \bn_\gamma \big )\\
&=\sum\limits_{\gamma\subset \partial ( C_\sigma\cap K )}  \hbbf_{\bn_\gamma }(\bu_\sigma, \bu_- ).
\end{split}
\label{fv:res:sigma}
\end{equation}
The last relation uses the consistency of the flux and the fact that $C_\sigma\cap K$ is a closed polygon. The quantity $\Phi_\sigma^K(\bu^h)$ is the normal flux on $C_\sigma\cap K$.
If now we sum up these three quantities and get:
\begin{equation*}
\begin{split}
\sum_{\sigma\in K} \Phi_\sigma^K(\bu_h)&= \bigg ( \hbbf_{\bn_{12}}(\bu_1,\bu_2)-\hbbf_{\bn_{13}}(\bu_1,\bu_3)-\bbf(\bu_1)\cdot\bn_{12}+\bbf(\bu_1)\cdot\bn_{31}\bigg )\\
&+\bigg ( \hbbf_{\bn_{23}}(\bu_2,\bu_3)-\hbbf_{\bn_{12}}(\bu_2,\bu_1)+\bbf(\bu_2)\cdot\bn_{12}-\bbf(\bu_2)\cdot\bn_{23}\bigg )\\
&+\bigg ( -\hbbf_{\bn_{23}}(\bu_3,\bu_2)+\hbbf_{\bn_{31}}(\bu_3,\bu_1)-\bbf(\bu_3)\cdot\bn_{23}+\bbf(\bu_3)\cdot\bn_{31}\bigg )\\
&= \bbf(\bu_1)\cdot \big ( \bn_{12}-\bn_{31}\big ) +\bbf(\bu_2)\cdot \big ( -\bn_{23}+\bn_{31}\big )
+\bbf(\bu_3)\cdot \big ( \bn_{31}-\bn_{23}\big )\\
&=\bbf(\bu_1)\cdot\frac{\bn_1}{2}+\bbf(\bu_2)\cdot\frac{\bn_2}{2}+\bbf(\bu_3)\cdot\frac{\bn_3}{2}
\end{split}
\end{equation*}
where $\bn_j$ is the scaled inward normal of the edge opposite to vertex $\sigma_j$, i.e. twice the gradient of the $\PP^1$ basis function
 $\varphi_{\sigma_j}$ associated to this degree of freedom.
Thus, we can reinterpret the sum as the boundary integral of the Lagrange interpolant of the flux.
The finite volume scheme is then a residual distribution scheme with residual defined by \eqref{fv:res:sigma}
and a total residual defined by
\begin{equation}
\label{fv:tot:residu}
\Phi^K:=\int_{\partial K} \bbf^h\cdot \bn , \qquad \bbf^h=\sum_{\sigma\in K} \bbf(\bu_\sigma)\varphi_\sigma.
\end{equation}

\begin{figure}[h]
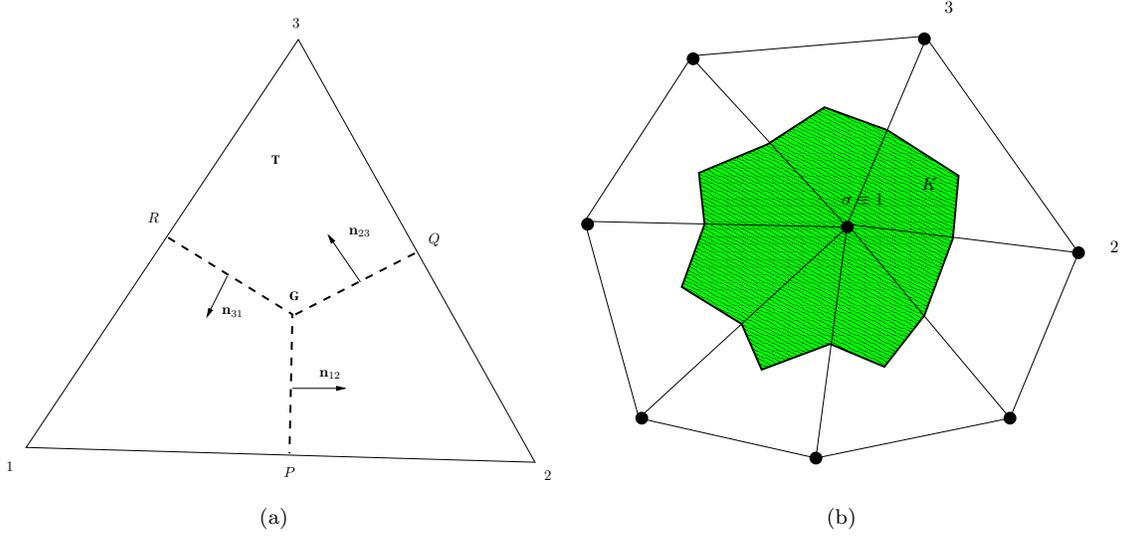

\begin{center}
\subfigure[]{\includegraphics[width=0.45\textwidth]{Figs/triangleP1_bon.pdf}}
\subfigure[]{\includegraphics[width=0.45\textwidth]{Figs/controlfv_bon.pdf}}
\end{center}
\caption{\label{fig:fv} Notations for the finite volume schemes. On the left: definition of the control volume for the degree of freedom $\sigma$.
 The vertex $\sigma$ plays the role of the vertex $1$ on the left picture for the triangle K. The control volume $C_\sigma$ associated to $\sigma=1$ is green on the right and corresponds to $1PGR$ on the left. The vectors $\bn_{ij}$ are normal to the internal edges scaled by the corresponding edge length}
\end{figure}

Let $K$ be a fixed triangle. We are given a set of residues $\{\Phi_\sigma^K\}_{\sigma\in K}$, our aim here is to define a
 flux function such that relations similar to \eqref{fv:res:sigma} hold true. We explicitly give the formula 
for $\PP^1$ and $\PP^2$ interpolant.

\subsection{General case}
One can deal with the general case, i.e when $K$ is a polytope contained in $\R^d$ with degrees of freedoms on the boundary of $K$. The set $\mathcal{S}$ is the set of degrees of freedom. We consider a triangulation $\mathcal{T}_K$ of $K$ whose vertices  are exactly the elements of $\mathcal{S}$. Choosing an orientation of $K$, it is propagated on $\mathcal{T}_K$: the edges are oriented.

The problem is to find quantities $\hbbf_{\bn_{\sigma,\sigma'}}$ for any edge $[\sigma,\sigma']$ of   $\mathcal{T}_K$ such that:
\begin{subequations}\label{GC:1}
\begin{equation}
\label{GC:1.1}
\Phi_\sigma=\sum_{\text{ edges }[\sigma,\sigma']} \hbbf_{\bn_{\sigma,\sigma'}}+\hbbf_\sigma^{b}
\end{equation}
with
\begin{equation}
\hbbf_{\bn_{\sigma,\sigma'}}=-\hbbf_{\bn_{\sigma',\sigma}}
\label{GC:1.2}
\end{equation}
and
\begin{equation}
\label{BC:1.3}
\hbbf_\sigma^b=\oint_{\partial K} \varphi_\sigma\; \hbbf_\bn (\bu^h,\bu^{h}_-) \; d\gamma.
\end{equation}
{The control volumes will be defined by their normals so that we get consistency.} The discontinuous nature of $\bu^h$, if any, is handled through the numerical flux $\hbbf_\bn$.

Note that \eqref{GC:1.2} implies the conservation relation
\begin{equation}
\label{GC:conservation}
\sum\limits_{\sigma\in K}\Phi_\sigma=\sum\limits_{\sigma\in K}\hbbf_\sigma^b,
\end{equation}
so that we rewrite \eqref{GC:1.1}, introducint $\Psi_\sigma=\Phi_\sigma-\hbbf_\sigma^b$ as:
\begin{equation}
\label{GC:1.1.bis}
\Psi_\sigma=\sum_{\text{ edges }[\sigma,\sigma']} \hbbf_{\bn_{\sigma,\sigma'}}
\end{equation}
\end{subequations}

\paragraph{The general  $\PP^1$ case.}
A solution is 
\begin{equation*}
\begin{split}
\hbbf_{\bn_{13}}&=\frac{\Psi_1-\Psi_3}{3} \\ 
\hbbf_{\bn_{23}}&=\frac{\Psi_2-\Psi_3}{3} \\ 
\hbbf_{\bn_{32}}&=\frac{\Psi_3-\Psi_2}{3}
\end{split}\end{equation*}

In order to describe the control volumes, we first have to make precise the normals $\bn_\sigma$ in that case. It is easy to see that in all the cases described above, we have 
$$\normal_\sigma=-\frac{\bn_\sigma}{2}.$$ Then a short calculation shows that
$$\begin{pmatrix}
\bn_{12} \\ \bn_{23} \\ \bn_{31} \end{pmatrix}=
\frac{1}{6}\begin{pmatrix} \bn_1-\bn_2 \\ \bn_2 -\bn_3 \\ \bn_3 -\bn_1 \end{pmatrix}.
$$
Using elementary geometry of the triangle, we see that these  are the normals of the elements of the dual mesh. For example, the normal $\bn_{12}$ is the normal of 
 $PG$, see figure \ref{fig:fv}.
 
 Relying more on the geometrical interpretation (once we know the control volumes), we can recover the same formula by elementary calculations, see \cite{icm}.

\paragraph{The general example of the $\PP^2$ approximation. }
Using a similar method, we get (see figure \ref{fig:P2} for some notations):
$$
\begin{array}{lcl}
\hbbf_{14}&=&\dfrac{1}{12}\big (\Psi_1-\Psi_4\big )+\dfrac{1}{36}\big ( \Psi_6-\Psi_5\big )+\dfrac {7}{36}\big (\Psi_1- \Psi_2\big )+\dfrac {5 }{36}\big (\Psi_3-\Psi_1\big )\\
&\\
\hbbf_{16}&=&\dfrac{1}{12}\big ( \Psi_4-\Psi_1\big )+\dfrac {5}{36}\big ( \Psi_5-\Psi_1)
+\dfrac {7}{36}\big ( \Psi_6-\Psi_1\big ) +\dfrac{1}{36}\big ( \Psi_3- \Psi_2\big )\\
&\\
\hbbf_{46}&=&\dfrac{2}{9}\big (\Psi_2-\Psi_6\big )+\dfrac{1}{9}\big (  \Psi_3- \Psi_5\big )\\
&\\
\hbbf_{54}&=&\dfrac{2}{9}\big (  \Psi_5-\Psi_2\big )+\dfrac{1}{9}\big ( \Psi_5-\Psi_1\big )\\
\end{array}
$$
$$
\begin{array}{lcl}
\hbbf_{42}&=&\dfrac {7}{36}\big (\Psi_2-\Psi_3\big ) +\dfrac{5}{36}\big (\Psi_1-\Psi_3\big )+\dfrac{1}{12}\big(\Psi_6-\Psi_3\big )+\dfrac{1}{36}\big (\Psi_5-\Psi_4\big ) \\

&\\
\hbbf_{25}&=&\dfrac{1}{36}\big (\Psi_2-\Psi_1\big )+\dfrac{5}{36}\big (\Psi_3-\Psi_5\big ) +\dfrac{7}{36}\big ( \Psi_3-\Psi_5\big )+\dfrac{1}{12}\big (\Psi_3-\Psi_6\big )  \\

&\\
\hbbf_{53}&=&\dfrac{1}{36}\big (\Psi_1-\Psi_6\big )+\dfrac{5}{36}\big (\Psi_3-\Psi_5\big )+\dfrac{7}{36}\big (\Psi_4-\Psi_5\big )+\dfrac{1}{12}\big (\Psi_2-\Psi_5\big )
\\
&\\
\hbbf_{63}&=& \dfrac{1}{36}\big (\Psi_4-\Psi_3\big )+\dfrac{5}{36}\big (\Psi_5-\Psi_1\big )+\dfrac{7}{36}\big (\Psi_5-\Psi_6\big )+\dfrac{1}{12}\big (\Psi_5-\Psi_2\big )\\
&\\

\hbbf_{65}&=&\dfrac{1}{9}\big (\Psi_1- \Psi_3\big )+\dfrac{2}{9}\big ( \Psi_6- \Psi_4\big )\end {array}
$$
Then we choose the boundary flux:
$$\hbbf_\sigma^b=\int_{\partial K}\varphi_\sigma\bn\; d\gamma$$ and get:
$$
\begin{array}{lll}
\normal_l=-\dfrac{\bn_l}{6} & \text{if }l=1,2,3\\ &&\\
\normal_4=\dfrac{\bn_3}{3}& \normal_5=\dfrac{\bn_1}{3}& \normal_6=\dfrac{\bn_2}{3}
\end{array}
$$
The normals are given by:
$$
\begin{array}{lcl}
\bn_{14}&=&\dfrac{1}{12}\big (\normal_1-\normal_4\big )+\dfrac{1}{36}\big ( \normal_6-\normal_5\big )+\dfrac {7}{36}\big (\normal_1- \normal_2\big )+\dfrac {5 }{36}\big (\normal_3-\normal_1\big )\\
&\\
\bn_{16}&=&\dfrac{1}{12}\big ( \normal_4-\normal_1\big )+\dfrac {5}{36}\big ( \normal_5-\normal_1)
+\dfrac {7}{36}\big ( \normal_6-\normal_1\big ) +\dfrac{1}{36}\big ( \normal_3- \normal_2\big )\\
&\\

\bn_{46}&=&\dfrac{2}{9}\big (\normal_2-\normal_6\big )+\dfrac{1}{9}\big (  \normal_3- \normal_5\big )\\
&\\
\bn_{54}&=&\dfrac{2}{9}\big (  \normal_5-\normal_2\big )+\dfrac{1}{9}\big ( \normal_5-\normal_1\big )
\end{array}
$$
$$
\begin{array}{lcl}
\bn_{42}&=&\dfrac {7}{36}\big (\normal_2-\normal_3\big ) +\dfrac{5}{36}\big (\normal_1-\normal_3\big )+\dfrac{1}{12}\big(\normal_6-\normal_3\big )+\dfrac{1}{36}\big (\normal_5-\normal_4\big ) \\

&\\
\bn_{25}&=&\dfrac{1}{36}\big (\normal_2-\normal_1\big )+\dfrac{5}{36}\big (\normal_3-\normal_5\big ) +\dfrac{7}{36}\big ( \normal_3-\normal_5\big )+\dfrac{1}{12}\big (\normal_3-\normal_6\big )  \\

&\\
\bn_{53}&=&\dfrac{1}{36}\big (\normal_1-\normal_6\big )+\dfrac{5}{36}\big (\normal_3-\normal_5\big )+\dfrac{7}{36}\big (\normal_4-\normal_5\big )+\dfrac{1}{12}\big (\normal_2-\normal_5\big )
\\
&\\
\bn_{63}&=& \dfrac{1}{36}\big (\normal_4-\normal_3\big )+\dfrac{5}{36}\big (\normal_5-\normal_1\big )+\dfrac{7}{36}\big (\normal_5-\normal_6\big )+\dfrac{1}{12}\big (\normal_5-\normal_2\big )\\
&\\

\bn_{65}&=&\dfrac{1}{9}\big (\normal_1- \normal_3\big )+\dfrac{2}{9}\big ( \normal_6- \normal_4\big )\end {array}
$$

\bigskip

There is not uniqueness, and it is possible to construct different solutions to the problem. In what follows, we show another possible construction. 
We consider the set-up defined by Figure \ref{fig:P2}.
\begin{figure}[h!]
\begin{center}
\includegraphics[width=0.45\textwidth]{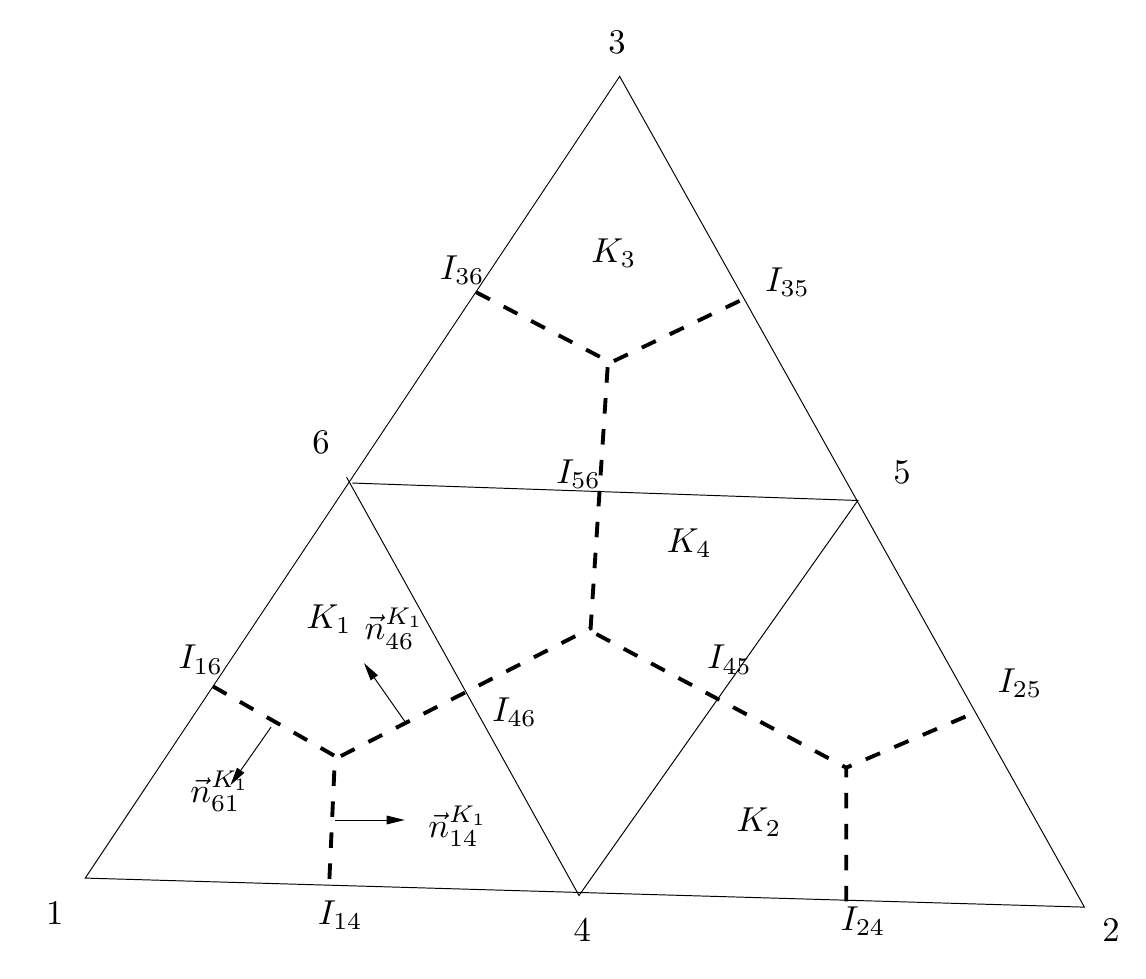}
\end{center}
\caption{\label{fig:P2} Geometrical elements for the $\PP^2$ case. $I_{ij}$ is the mid-point between the vertices $i$ and $j$. The intersections of the dotted lines are the centroids of the sub-elements.}
\end{figure}
}
\section{Proof of \eqref{algebre2}.}\label{proof}
\begin{proof}[Proof of \eqref{algebre2}]
We start from \eqref{RD:scheme} which is multiplied by $\bv_\sigma$, and these relations are added for each $\sigma\in \mathcal{S}$. We get:
$$
0=\sum\limits_{\sigma\in \mathcal{S}}\bv_\sigma\bigg ( \sum\limits_{K\subset \Omega, \sigma\in K} \Phi_\sigma^K(\bu^h)+\sum\limits_{\Gamma\subset\partial \Omega, \sigma\in \Gamma}\Phi_\sigma^\Gamma(\bu^h,\bu_b) \bigg ).
$$
Permuting the sums on $\sigma$ and $K$, then on $\sigma$ and $\Gamma$, we get:
$$
0=\sum\limits_{K\subset \Omega} \bigg ( \sum\limits_{\sigma\in K} \bv_\sigma\Phi_\sigma^K(\bu^h)\bigg ) + \sum\limits_{\Gamma\subset\partial \Omega}\bigg ( \sum\limits_{\sigma\in \Gamma} \bv_\sigma\Phi_\sigma^\Gamma(\bu^h,\bu_b)\bigg ) .$$
We look at the first term, the second is done similarly.
We have, introducing $\Phi_\sigma^{K, Gal}$ and $\#K$ the number of degrees of freedom in $K$,
\begin{equation*}
\begin{split}
\sum\limits_{\sigma\in K} \bv_\sigma\Phi_\sigma^K(\bu^h)&=\sum\limits_{\sigma\in K} \bv_\sigma \Phi_\sigma^{K, Gal}(\bu^h) +\sum\limits_{\sigma\in K} \bv_\sigma\bigg ( \Phi_\sigma^K(\bu^h)-\Phi_\sigma^{K, Gal}(\bu^h) \bigg )\\
&= -\oint_K \nabla \bv_h \cdot \bbf(\bu^h) \; d\bx +\oint_{\partial K} \bv^h \hbbf_\bn(\bu^h_\Kp, \bu^{h}_\Km)\; d\gamma+ \sum\limits_{\sigma\in K} \bv_\sigma\bigg ( \Phi_\sigma^K(\bu^h)-\Phi_\sigma^{K, Gal}(\bu^h) \bigg )\\
&=-\oint_K \nabla \bv_h \cdot \bbf(\bu^h) \; d\bx +\oint_{\partial K} \bv^h \hbbf_\bn(\bu^h_\Kp, \bu^{h}_\Km)\; d\gamma\\
&\qquad \qquad \qquad+ \frac{1}{\#K}\sum\limits_{\sigma,\sigma'\in K} (\bv_\sigma-\bv_{\sigma'})\bigg ( \Phi_\sigma^K(\bu^h)-\Phi_\sigma^{K, Gal}(\bu^h) \bigg )
\end{split}
\end{equation*}
because from \eqref{conservation:K},  $$\sum\limits_{\sigma\in K} \big ( \Phi_\sigma^K(\bu^h)-\Phi_\sigma^{K, Gal}(\bu^h) \big )=0.$$

Similarly, we have
\begin{equation*}
\begin{split}
\sum\limits_{\sigma\in \Gamma } \bv_\sigma\Phi_\sigma^\Gamma(\bu^h)&=\oint_{\Gamma}\bv^h \big (\hat{\bbf}_\bn(\bu^h,\bu_b)-\bbf(\bu^h)\cdot \bn\big ) \; d\gamma+\sum\frac{1}{\#\Gamma}\sum\limits_{\sigma,\sigma'\in \Gamma} (\bv_\sigma-\bv_{\sigma'})(\Phi_\sigma^\Gamma \big (\bu^h,\bu_b)-\Phi_\sigma^{Gal,\Gamma}(\bu^h,\bu_b)\big )
\end{split}
\end{equation*}
Adding all the relations and using the assumption \ref{assump:quad} on the quadrature formulas, we get:
\begin{equation*}
\begin{split}
0&= \sum\limits_{K\subset \Omega} \bigg ( -\oint_K \nabla \bv_h \cdot \bbf(\bu^h) \; d\bx +\oint_{\partial K} \bv^h \hbbf_\bn(\bu^h_\Kp, \bu^{h}_\Km)\; d\gamma\bigg ) +\sum\limits_{\Gamma\subset \partial \Omega}\oint_{\Gamma}\bv^h \big (\hat{\bbf}_\bn(\bu^h,\bu_b)-\bbf(\bu^h)\cdot \bn\big ) \; d\gamma\\
& \qquad  +\sum\limits_{K\subset \Omega}\frac{1}{\#K}\bigg ( \sum\limits_{\sigma,\sigma'\in K} (\bv_\sigma-\bv_{\sigma'})\bigg ( \Phi_\sigma^K(\bu^h)-\Phi_\sigma^{K, Gal}(\bu^h) \bigg ) \bigg )\\
&\qquad \qquad
+\sum\limits_{\Gamma\subset \partial \Omega} \frac{1}{\#\Gamma}\bigg ( \sum\limits_{\sigma,\sigma'\in \Gamma} (\bv_\sigma-\bv_{\sigma'})(\Phi_\sigma^\Gamma \big (\bu^h,\bu_b)-\Phi_\sigma^{Gal,\Gamma}(\bu^h,\bu_b)\big )\bigg )
\end{split}
\end{equation*}
i.e. after having defined $[\bv^h]= \bv^h_\Kp-\bv_\Km^{h}$ and chosen one orientation of the internal edges $e\in \mathcal{E}_h$,  we get \eqref{algebre2}.
\end{proof}

\begin{proof}[Proof of Proposition \ref{prop:approx}]
We first show that $\Phi_\sigma^{K, Gal}(\pi_h(\bu))=\mathcal{O}\bigl(h^{k+d}\bigr)$. Since $\bu$ is regular enough, we have pointwise $\text{ div } \bbf(\bu) =0$ on $K$, so that, by consistency of the flux, 
$$
0=-\int_K \nabla \varphi\cdot \bbf(\bu)\; d\bx+\int_{\partial K} \varphi \hbbf_\bn(\bu, \bu)\; d\gamma.
$$
This shows that 
$$-\oint_K \nabla \varphi\cdot \bbf(\bu)\; d\bx+\oint_{\partial K} \varphi \hbbf_\bn(\bu, \bu)\; d\gamma=O(h^{k+d+1}).$$
Then,
\begin{equation*}
\begin{split}
\Phi_\sigma^{K, Gal}(\pi_h(\bu))&=-\oint_K\nabla\varphi_\sigma\cdot \big (\bbf(\pi_h(\bu))-\bbf(\bu)\big ) \; d\bx +\oint_{\partial K}\varphi_\sigma\big ( \hbbf_\bn(\pi_h(\bu), \pi_h(\bu))-\hbbf_\bn(\bu, \bu)\big )\; d\gamma\\
&\qquad \qquad \qquad+\oint_K \nabla \varphi\cdot \bbf(\bu)\; d\bx-\oint_{\partial K} \varphi \hbbf_\bn(\bu, \bu)\; d\gamma \\
& = |K| \times \mathcal{O}(h^{-1}) \times \mathcal{O}(h^{k+1}) + |\partial K|\times  \mathcal{O}(1 )\mathcal{O}(h^{k+1})\\
&  \qquad \qquad \qquad +O(h^{k+d+1})\\
&=\mathcal{O}(h^{d})\times \mathcal{O}(h^{-1}) \times \mathcal{O}(h^{k+1})+\mathcal{O}(h^{d-1})\times  \mathcal{O}(1 )\times\mathcal{O}(h^{k+1})\\
& \qquad \qquad \qquad + O(h^{k+d+1})\\
&=\mathcal{O}(h^{d+k})
\end{split}
\end{equation*}
because the flux is Lipschitz continuous and the mesh is regular.

The result on the boundary term is similar since  the boundary numerical flux is upwind and the boundary of $\Omega$ is not characteristic: only two types of boundary faces exists, the upwind and downwind ones. On the downwind faces, the boundary flux vanishes. On the upwind ones, we get the estimate  for the Galerkin boundary residuals thanks to the  same approximation argument.

The mesh is assumed to be regular: the number of elements (resp. edges) is $O(h^{-d})$ (resp. $O(h^{-d+1})$). 
Let us assume \eqref{eq: residual accuracy}. Let $\bv\in C_0^1(\overline{\Omega})$. Using \eqref{algebre2} for $\pi_h(\bu)$,\begin{equation*}
\begin{split}
\mathcal{E}\bigl(\bu^h, \varphi\bigr) =&  -\oint_\Omega \nabla \pi_h(v) \cdot \bbf(\pi_h(\bu)) \; d\bx  +\oint_{\partial \Omega}\pi_h(v) \big (\hat{\bbf}_\bn(\pi_h(\bu),\bu_b)-\bbf(\pi_h(\bu))\cdot \bn\big ) \; d\gamma\\
&\qquad +\sum\limits_{e\in \mathcal{E}_h} \oint_e[\pi_h(v)]\hbbf_\bn(\pi_h(\bu),\pi_h(\bu)_{-})\; d\gamma+\sum\limits_{K\subset \Omega}\frac{1}{\#K}\bigg ( \sum\limits_{\sigma,\sigma'\in K} (\bv_\sigma-\bv_{\sigma'})\bigg ( \Phi_\sigma^K(\pi_h(\bu))-\Phi_\sigma^{K, Gal}(\pi_h(\bu)) \bigg ) \bigg )\\
&\qquad \qquad
+\sum\limits_{\Gamma\subset \partial \Omega} \frac{1}{\#\Gamma}\bigg ( \sum\limits_{\sigma,\sigma'\in \Gamma} (\bv_\sigma-\bv_{\sigma'})(\Phi_\sigma^\Gamma \big (\pi_h(\bu),\bu_b)-\Phi_\sigma^{Gal,\Gamma}(\pi_h(\bu),\bu_b)\big )\bigg ).
\end{split}
\end{equation*}
 where $\pi_h(\bu)_{-}$ represents the interpolant of $\bu$ on $K^-$.

We have, using 
$$-\oint_\Omega \nabla \pi_h(v)\cdot \bbf(\bu)\; d\bx+\oint_{\partial \Omega} \pi_h(v) \big (\mathcal{F}_\bn(\bu,\bu_b)-\bbf(\bu)\cdot \bn\big ) \; d\gamma=0,$$
\begin{equation*}
\begin{split}
-\oint_\Omega \nabla \pi_h(v) &\cdot \bbf(\pi_h(\bu)) \; d\bx  +\oint_{\partial \Omega}\pi_h(v) \big (\hat{\bbf}_\bn(\pi_h(\bu),\bu_b)-\bbf(\pi_h(\bu))\cdot \bn\big ) \; d\gamma\\&=
-\oint_\Omega \nabla \pi_h(v) \cdot \big ( \bbf(\pi_h(\bu))-\bbf(\bu)\big ) \; d\bx
+\oint_{\partial \Omega}\pi_h(v) \big ( \hbbf_\bn(\pi_h(\bu),\bu_b)-\hbbf_\bn(\bu,\bu_b)\big )\; d\gamma\\&\qquad \qquad - \oint_{\partial \Omega}\pi_h(v)\big( \bbf(\pi_h(\bu))-\bbf(\bu) \big )\cdot \bn \; d\gamma\\
&=\mathcal{O}(h^{k+1})
\end{split}
\end{equation*}
since the flux on the boundary is the upwind flux $\mathcal{F}_\bn$, and using  the approximation properties of $\pi_h(\bu)$.

Then
$$\sum\limits_{e\in \mathcal{E}_h} \oint_e[\pi_h(v)]\hbbf_\bn(\pi_h(\bu),\pi_h(\bu)^{-})\; d\gamma=\mathcal{O}(h^{-d+1}) \times \mathcal{O}(h^{d-1})\times  \mathcal{O}(h^{k+1})\times \mathcal{O}(1)=\mathcal{O}(h^{k+1}),$$
$$\sum\limits_{K\subset \Omega}\frac{1}{\#K}\bigg ( \sum\limits_{\sigma,\sigma'\in K} (\bv_\sigma-\bv_{\sigma'})\bigg ( \Phi_\sigma^K(\pi_h(\bu))-\Phi_\sigma^{K, Gal}(\pi_h(\bu)) \bigg ) \bigg )= \mathcal{O}(h^{-d})\times  \mathcal{O}(h) \times  \mathcal{O}(h^{k+d})=\mathcal{O}(h^{k+1}),$$
and similarly
$$\sum\limits_{\Gamma\subset \partial \Omega} \frac{1}{\#\Gamma}\bigg ( \sum\limits_{\sigma,\sigma'\in \Gamma} (\bv_\sigma-\bv_{\sigma'})(\Phi_\sigma^\Gamma \big (\pi_h(\bu),\bu_b)-\Phi_\sigma^{Gal,\Gamma}(\pi_h(\bu),\bu_b)\big )\bigg )=\mathcal{O}(h^{-d+1})\times\mathcal{O}(h) \times  \mathcal{O}(h^{k+d-1})=\mathcal{O}(h^{k+1})$$
thanks to the regularity of the mesh, that $\pi_h(v)$ is the interpolant of a $C^1$ function and the previous estimates.
\end{proof}

\end{document}